\documentclass[3p]{elsarticle}

\usepackage{bm}
\usepackage[parfill]{parskip}
\usepackage[utf8]{inputenc}
\usepackage{url}
\usepackage{amsmath,amssymb,amsfonts,amsthm}
\usepackage{color}
\usepackage{bm}
\usepackage{graphicx}
\usepackage{subcaption}
\usepackage{lipsum}
\usepackage{amsfonts}
\usepackage{graphicx}
\usepackage{epstopdf}

\usepackage{amsmath}
\usepackage{algorithm}
\usepackage{algorithmic}

\ifpdf
  \DeclareGraphicsExtensions{.eps,.pdf,.png,.jpg}
\else
  \DeclareGraphicsExtensions{.eps}
\fi

\newtheorem{theorem}{Theorem}

\DeclareMathOperator*{\argmin}{arg\,min}

\makeatletter
\newsavebox\myboxA
\newsavebox\myboxB
\newlength\mylenA

\newcommand*\xoverline[2][0.75]{%
    \sbox{\myboxA}{$\m@th#2$}%
    \setbox\myboxB\null
    \ht\myboxB=\ht\myboxA%
    \dp\myboxB=\dp\myboxA%
    \wd\myboxB=#1\wd\myboxA
    \sbox\myboxB{$\m@th\overline{\copy\myboxB}$}
    \setlength\mylenA{\the\wd\myboxA}
    \addtolength\mylenA{-\the\wd\myboxB}%
    \ifdim\wd\myboxB<\wd\myboxA%
       \rlap{\hskip 0.5\mylenA\usebox\myboxB}{\usebox\myboxA}%
    \else
        \hskip -0.5\mylenA\rlap{\usebox\myboxA}{\hskip 
0.5\mylenA\usebox\myboxB}%
    \fi}
\makeatother

\journal{arXiv.org}

\newcommand{\TheTitle}{A Realizable Filtered Intrusive Polynomial Moment Method} 

\date{\today}

\usepackage{amsopn}
\DeclareMathOperator{\diag}{diag}

\definecolor{darkgreen}{rgb}{0,0.6,0}

\def\R{\mathbb{R}}
\def\cR{\mathcal{R}}
\def\uhat{\bm{\hat u}}

\def\cL{\mathcal{L}}
\def\cR{\mathcal{R}}

\newcommand{\vint}[1]{\langle #1 \rangle}
\newcommand{\Vint}[1]{\left\langle #1 \right\rangle}

\begin{document}

\begin{frontmatter}

\title{\TheTitle}

\author[adressGraham]{Graham Alldredge}
\author[adressMartin]{Martin Frank}
\author[adressJonas]{Jonas Kusch}
\author[adressRyan]{Ryan McClarren}

\address[adressGraham]{Berlin, gwak@posteo.net}
\address[adressMartin]{Karlsruhe Institute of Technology, Karlsruhe, 
martin.frank@kit.edu}
\address[adressJonas]{Karlsruhe Institute of Technology, Karlsruhe,
    jonas.kusch@kit.edu}
\address[adressRyan]{University of Notre Dame, Notre Dame,  rmcclarr@nd.edu}

\begin{abstract}
Intrusive uncertainty quantification methods for hyperbolic problems exhibit spurious oscillations at shocks, which leads to a significant reduction of the overall approximation quality. Furthermore, a challenging task is to preserve hyperbolicity of the gPC moment system. An intrusive method which guarantees hyperbolicity is the intrusive polynomial moment (IPM) method, which performs the gPC expansion on the entropy variables. The method, while still being subject to oscillations, requires solving a convex optimization problem in every spatial cell and every time step.

The aim of this work is to mitigate oscillations in the IPM solution by applying filters. Filters reduce oscillations by damping high order gPC coefficients. Naive filtering, however, may lead to unrealizable moments, which means that the IPM optimization problem does not have a solution and the method breaks down. In this paper, we propose and analyze two separate strategies to guarantee the existence of a solution to the IPM problem. First, we propose a filter which maintains realizability by being constructed from an underlying Fokker-Planck equation. Second, we regularize the IPM optimization problem to be able to cope with non-realizable gPC coefficients. Consequently, standard filters can be applied to the regularized IPM method. We demonstrate numerical results for the two strategies by investigating the Euler equations with uncertain shock structures in one- and two-dimensional spatial settings. We are able to show a significant reduction of spurious oscillations by the proposed filters.
\end{abstract}

\begin{keyword}
uncertainty quantification, conservation laws, hyperbolic, intrusive, entropy, 
filtering, realizability
\end{keyword}

\end{frontmatter}

\section{Introduction}
Conservation laws describe the behavior of a large number of physical applications; these laws can be found in fields as diverse  as fluid dynamics, radiation transport or plasma physics. Given the desire to use computational simulations to predict the behavior of such systems, quantifying the effect of stochastic uncertainties, such as measurement errors or modeling assumptions, is an important task. With this goal in mind we consider a general conservation law
\begin{subequations}\label{eq:hyperbolicProblem}
\begin{align}
\partial_t \bm{u}(t,\bm{x},\bm{\xi}) + 
\nabla&\cdot\bm{f}(\bm{u}(t,\bm{x},\bm{\xi})) = \bm{0} \enskip \text{ in } D, 
\\ 
\label{eq:ic}
\bm{u}(t=0,\bm{x},&\bm{\xi}) = \bm{u}_{\text{IC}}(\bm{x},\bm{\xi})
\end{align}
\end{subequations}
for the state variable $\bm u\in\mathbb{R}^m$ at time $t\in\mathbb{R}^+$ and 
spatial position $\bm{x}\in D\subseteq \mathbb{R}^d$ and also dependent on the
random variable $\bm{\xi}\in\Theta\subseteq\mathbb{R}^p$ with 
probability density function $f_{\Xi}(\bm{\xi})$.
Supplemented with adequate boundary conditions, these equations can be solved 
using various techniques \cite{mcclarren2018uncertainty}. In this work, we 
focus 
on intrusive methods for the time evolution of stochastic quantities such as 
expected value or variance. These quantities are more generally the 
moments
\begin{align}
 \int_\Theta \varphi_i(\bm{\xi}) \bm{u}(t, \bm{x}, \bm{\xi})
  f_{\Xi}(\bm{\xi}) d \bm{\xi}
\end{align}
of the random vector $\bm{u}$ with respect to the basis functions $\varphi_i$.
In the sequel we use bracket notation for integration over $\Theta$ against
$f_\Xi$:
\begin{align}
 \Vint{\cdot} := \int_\Theta \cdot f_{\Xi}(\bm{\xi}) d \bm{\xi}
\end{align}
and assume orthonormal basis functions. We use general polynomial chaos (gPC) polynomials \cite{wiener1938homogeneous,xiu2002wiener}. Hence the basis functions $\varphi_k:\Theta\rightarrow\mathbb{R}$ fulfill
\begin{align}
 \Vint{\varphi_k \varphi_{k'}} = \delta_{k, k'}
  \qquad \text{for all } k, k'.
\end{align}

For simplicity of exposition, let us start with a one-dimensional uncertainty 
($p = 1$) and consider polynomial basis functions $\varphi_i$, where $i$ 
corresponds to the degree of $\varphi_i$.
Various intrusive methods can be used to approximate the moments, and they all 
come with certain advantages and shortcomings.
Generally, intrusive methods specify a set of equations describing the time 
evolution of the moments, by first testing the original equations 
\eqref{eq:hyperbolicProblem} with a finite set of the basis functions 
$\varphi_i$, say $i\in\{0,\dots,N\}$ leading to
\begin{align}\label{eq:momentEquations}
\partial_t\Vint{\varphi_i\bm{u}} + 
\nabla&\cdot\langle\varphi_i\bm{f}(\bm{u}(t,\bm{x},\cdot))\rangle = 0
\enskip \text{ in } D,\, i \in \{0,\dots,N\}.
\end{align}
This is not a closed system of equations because, due to the nonlinearity of $f$, 
$\vint{\varphi_i\bm{f}(\bm u(t,\bm{x},\cdot))}$ in general cannot be expressed in terms 
of the moments $\{\vint{\varphi_0\bm{u}}, \dots , \vint{\varphi_i\bm{u}}\}$.
The system is closed by providing an ansatz 
$\mathcal{U}(\bm{\hat u}_0,\cdots,\bm{\hat u}_N)\simeq \bm{u}$ for the 
solution inside the flux, where $\bm{\hat u}_i \simeq \vint{\varphi_i\bm{u}}$
for each $i$.
After inserting this closure into \eqref{eq:momentEquations} and collecting 
the moments and basis functions into the matrix
$\bm{\hat u}=(\bm{\hat u}_0,\dots,\bm{\hat u}_N)^T \in \R^{(N + 1) \times m}$ 
and vector $\bm{\varphi}:=(\varphi_0,\dots,\varphi_N)^T\in\mathbb{R}^{N+1}$, we 
have the approximate system
\begin{subequations}\label{eq:momentEquationsClosed}
\begin{align}
 \partial_t\bm{\hat u} + 
  \nabla\cdot\Vint{\bm\varphi\bm{f}(\mathcal{U}(\bm{\hat u}))} &= \bm 0
 \enskip \text{ in } D, \\ \label{eq:icMoments}
 \bm{\hat u}(t=0,\bm{x})
  =& \Vint{\bm\varphi \bm{u}_{\text{IC}}(\bm{x},\cdot)}.
\end{align}
\end{subequations}
Note that a similar need to construct a closure also appears in fields such as kinetic theory, which is closely related to uncertainty quantification \cite{kusch2018intrusive}. The choice of $\mathcal{U}$ crucially affects properties of the resulting 
moment 
system. A straightforward closure choice makes use of the basis functions 
$\varphi_i$ to represent the solution:
\begin{align*}
\mathcal{U}_k = \sum_{i=0}^N \hat w_{ik} \varphi_i
 \qquad k \in \{1, \dots , m\},
\end{align*}
where $\hat w_{ik} \in \R$ are expansion coefficients which are at this 
point free to choose and $\bm{\hat w}\in\mathbb{R}^{(N + 1) \times m}$ is 
the matrix containing all the expansion coefficients.
Assume that a given moment vector belongs to a function $\bm{u}_{\text{ex}}\in\mathbb{R}^{m}$. First, we pick $\bm{\hat w}$ such that the L$^2$ distance between this function and the polynomial approximation is minimized, 
i.e., we solve
\begin{align}\label{eq:costFunctionSG}
 (\hat w_{0k}, \dots , \hat w_{Nk})
  = \argmin_{\bm{w} \in \R^{N + 1}} \mathcal{F}_k(\bm{w}),
 \qquad \text{ where }
 \mathcal{F}_k(\bm{w}) := \Vint{(\bm w^T \bm\varphi - 
  \bm{u}_{\text{ex}, k})^2},
 \qquad k \in \{1, \dots , m\}.
\end{align}
Since we are using orthonormal polynomials for basis functions, the resulting 
optimal expansion coefficients are simply the moments, i.e., 
$\bm{\hat w} \equiv \bm{\hat u}$.
The resulting ansatz, which gives the stochastic-Galerkin (SG) closure 
\cite{ghanem2003stochastic} is
\begin{align*}
\mathcal{U}_{SG}(\bm{\hat{u}}):=\bm{\hat{u}}^T\bm{\varphi}.
\end{align*}
Compared to other intrusive methods such as the intrusive polynomial moment (IPM) method \cite{poette2009uncertainty} or the Roe transformation method \cite{pettersson2014stochastic}, stochastic-Galerkin is computationally 
inexpensive, however suffers from a possible loss of hyperbolicity 
\cite{poette2009uncertainty}. This necessitates manipulating the solution to be 
able to run the method in certain settings \cite{schlachter2018hyperbolicity} or imposing a modification of the flux discretization \cite{dai2021hyperbolicity}.
Moreover, SG suffers from oscillations when the solution is not sufficiently 
smooth \cite{le2004uncertainty}, which is a common issue in hyperbolic 
conservation laws where the solution tends to form discontinuous shocks 
\cite{poette2009uncertainty}. These oscillations impair the solution quality and lead to physically incorrect solution values, which is why different strategies to mitigate oscillations have been developed. Oscillations dissolve when the number of chosen moments is increased \cite{pettersson2009numerical,offner2017stability}, which however adds computational costs. Additionaly, amplifying the numerical diffusion mitigates oscillations \cite{offner2017stability}, but at the same time leads to non-accurate solution approximations in deterministic regions. Furthermore, the use of multi-elements in the stochastic domain reduces artifacts which stem from oscillations \cite{Wan2006,Tryoen2010}. In \cite{kusch2020oscillation} multi-elements are used for the IPM method. The strategy reduces numerical artifacts in the IPM solution while leading to a significant reduction in computational costs. 

Another strategy for reducing oscillations in spectral methods is to apply a filter between time 
steps \cite{boyd2001chebyshev,hesthaven2007spectral}. Filters have for example been used in the context of kinetic theory \cite{mcclarren2010simulating,mcclarren2010robust,radice2013new,laboure2016implicit,frank2016convergence} as well as spatial discretization methods \cite{gassner2013accuracy,lundquist2020stable}. A filtered stochastic-Galerkin method has been proposed in \cite{kusch2018filtered}, yielding a significant mitigation of oscillations. Just as the classical stochastic-Galerkin method, this filtered SG method can potentially lead to non-physical solution values, leading to a failure of the method. One possible strategy to derive a filter is to replace the SG cost function \eqref{eq:costFunctionSG} by 
\begin{align}\label{eq:costFunctionfSG}
 \mathcal{F}^{(\lambda)}(\bm{w}_k; \bm{\hat u}_k)
  := \Vint{(\bm{w}_k^T \bm\varphi
  - \bm{\hat u}_k^T \bm\varphi)^2} + \lambda T(\bm{w}_k),
\end{align}
for each $k \in \{1, \dots , m\}$, where $\bm{\hat u}_k$ is the $k$-th column 
of $\bm{\hat u}$ and $T$ is a function which penalizes oscillations in the 
resulting ansatz $\bm{w}_j^T \bm\varphi$.
The parameter $\lambda$ is a user-determined filter strength.
As an example, the L$^2$ filter for uniform distributions which has been investigated in \cite{kusch2018filtered} uses
\begin{align}\label{eq:L2filter}
T(\bm{w}) = \left\langle 
\left(\cL\left(\bm{w}^T\bm{\varphi}\right)\right)^2 \right\rangle, \qquad 
\text{ with } \cL u := \frac{d}{d\xi}\left((1-\xi^2)\frac{d}{d\xi}u\right).
\end{align}
Note that for more general distributions with gPC polynomials $P_i:\Theta\rightarrow\R$, the operator $\cL$ should be chosen to be the eigenoperator the chosen polynomial basis. I.e., we have $\cL P_i = \mu_i P_i$, where $\mu_i$ are the corresponding eigenvalues. With this choice of $T$ the minimizer of
$\mathcal{F}^{(\lambda)}(\bm{w}_k; \bm{\hat u}_k)$ in 
\eqref{eq:costFunctionfSG} is given by the components
$\overline{w}_{ik} := g_i(\lambda)\hat u_{ik}$, where for uniform distributions, the filter function 
$g_i$ is
\begin{align}\label{eq:L2filtering}
g_i(\lambda):=\frac{1}{1+\lambda i^2(i+1)^2}.
\end{align}
Hence, the optimal expansion coefficients are dampened with increasing strength 
as the order $i$ increases. A more sophisticated filter function takes the form
\begin{align}\label{eq:filteredSGb}
\tilde g_i(\lambda) := h(i/N)^{\lambda \Delta t},
\end{align}
where $\Delta t$ is the time step size of the time discretization chosen to numerically solve the SG moment system \eqref{eq:momentEquationsClosed}. Compared to \eqref{eq:L2filtering}, this filter adapts with the chosen time step size to ensure that increasing or decreasing $\Delta t$ will not heavily weaken or amplify the effect of filtering. When $\Delta t\rightarrow 0$, the filtered equations tend towards a continuous limit \cite{radice2013new}.  Furthermore, a so-called filter order $\alpha$ can be introduced. This parameter aims at preserving the approximation order for sufficiently smooth functions, see e.g. \cite{boyd1996erfc}. One example for a filter of order of $\alpha$ is \begin{align}\label{eq:expFilter}
h(\zeta) = \exp\left(c\zeta^{\alpha}\right)
\end{align}
where $c = \log(\varepsilon_M)$ and $\varepsilon_M$ is the machine accuracy. This filter choice leads to the exponential filter \cite{hoskins1980representation}. Another choice, which is called the Erfc filter \cite{boyd1996erfc} uses
\begin{align}\label{eq:erfcFilter}
h(\zeta) = \frac{1}{2}\text{erfc}\left(2\alpha^{1/2}\left(|\zeta|-1/2\right)\right).
\end{align}
If we put the filter weights into the diagonal matrix
$L(\lambda) := \diag\{ \tilde g_i(\lambda) \}_{i = 0}^N$, we can write the 
solution as $\bm{\overline{\alpha}} = L(\lambda) \bm{\hat u}$.

Within the full algorithm, the filter is applied as follows.
First, a finite-volume method with forward-Euler time discretization for the 
moment equations \eqref{eq:momentEquationsClosed} with the stochastic-Galerkin 
closure is given by
\begin{align}
 \bm{\hat u}_j^{n+1} = \bm{\hat u}_j^{n} -
  \frac{\Delta t}{\Delta x}(\bm{F^*}(\bm{\hat u}_{j}^n,\bm{\hat u}_{j+1}^n)
  - \bm{F^*}(\bm{\hat u}_{j-1}^n,\bm{\hat u}_{j}^n)),
\end{align}
where $\bm{\hat u}_j^n$ is the matrix of moment components for spatial call $j$ 
and time step $n$ and $\bm{F^*}$ is the numerical flux function (see 
Section~\ref{sec:implementationSpaceTime} below for details).
The filtered method is then given by
\begin{subequations}\label{eq:filteredEquations}
\begin{align}
 \bm{\overline{u}}_j^n &= L(\lambda) \bm{\hat u}_j^n, \\
 \bm{\hat u}_j^{n+1} &= \bm{\overline{u}}_j^{n} -
 \frac{\Delta t}{\Delta x}(\bm{F^*}(\bm{\overline{u}}_{j}^n,
 \bm{\overline{u}}_{j+1}^n) - \bm{F^*}(\bm{\overline{u}}_{j-1}^n,
 \bm{\overline{u}}_{j}^n)).
\end{align}
\end{subequations}
We call this the filtered stochastic-Galerkin (fSG) method 
\cite{kusch2018filtered}, and it gives accurate, non-oscillatory results if an 
adequate filter strength is chosen. Furthermore, the ability to analytically 
determine the minimizer of \eqref{eq:costFunctionfSG} allows low computational 
costs and therefore fast runtimes.
Unfortunately, as with the original stochastic-Galerkin method, loss of 
hyperbolicity remains an issue, since the filter does not guarantee physically correct solution values. When for example taking a look at the Euler equations, the filter does not guarantee positivity of density and pressure. Hence, even though the filter tends to push the SG solution towards a physically feasible state, filtered stochastic-Galerkin can still generate physically incorrect solution values. This leads to a failure of the method, since various quantities such as the speed of sound will no longer be real numbers. One possibility of preserving physically correct solution bounds of the SG solution is to apply a hyperbolicity-preserving limiter to the SG method \cite{schlachter2018hyperbolicity}. A strategy to combine the use of such a limiter as well as filters has been presented in \cite{kusch2020oscillation}.

A generalization of stochastic-Galerkin, which can preserve hyperbolicity without the use of a limiter is 
the intrusive polynomial moment (IPM) method \cite{poette2009uncertainty}.
The main idea is to close the system by minimizing a mathematical entropy of 
the conservation law \eqref{eq:hyperbolicProblem} while fulfilling a moment 
constraint.
For a given convex entropy $s:\mathbb{R}^m\to\mathbb{R}$ for the original 
problem \eqref{eq:hyperbolicProblem}, this constrained optimization 
problem reads
\begin{align}\label{eq:primalProblem}
 \mathcal{U}(\bm{\hat u}) = \argmin_{\bm{u}} \Vint{s(\bm{u})}
  \enskip \text{ subject to }
 \bm{\hat u} = \Vint{\bm{\varphi} \bm{u}}
\end{align}
where $\bm{\varphi} \bm{u}$ is an outer product, i.e., 
$[\bm{\varphi} \bm{u}]_{ik} = \varphi_i u_k$.
Using duality, this constrained optimization problem can be transformed into the 
unconstrained finite-dimensional problem
\begin{align}\label{eq:dualProblem}
 \min_{\bm{\hat w} \in \mathbb{R}^{(N+1) \times m}}
  \left\{ \Vint{s_*(\bm{\hat w}^T \bm\varphi)}
  - \bm{\hat w} \cdot \bm{\hat u} \right\},
\end{align}
where $s_*:\mathbb{R}^m\to\mathbb{R}$ is the Legendre transformation of $s$
and
\begin{align}
 \bm{\hat w} \cdot \bm{\hat u} = \sum_{i = 0}^N \sum_{k = 1}^m
  \hat{w}_{ik} \hat{u}_{ik}.
\end{align}
If we let $\bm{\hat v}(\bm{\hat u})$ denote the minimizer of the dual 
problem \eqref{eq:dualProblem}, the minimizer of the primal problem 
\eqref{eq:primalProblem} is given by
\begin{align}\label{eq:ansatz}
 \mathcal{U}(\bm{\hat u}) = \left( s' \right)^{-1}(\bm{\hat{v}}(\bm{\hat u})^T 
\bm{\varphi}).
\end{align}
Here we use that $s_{*}'(v) = \left( s' \right)^{-1}(v)$. It is important to point out that if the inverse of $s'$ only takes values in a set $\cR_{\bm{u}}\subset\mathbb{R}^m$, the presented optimization problem is only solvable if the moments $\bm{\hat u}$ lie inside the so-called realizable set
\begin{align}\label{eq:realizableSetGeneral}
 \cR := \left\{ \uhat \in \R^{(N + 1)\times m} : 
  \exists \bm{u} = \bm{u}(\xi) \in \cR_{\bm{u}}
  \text{ for all } \xi \in \Theta \text{ such that }
  \uhat = \Vint{\bm{\varphi} \bm{u}} \right\}.
\end{align}
A detailed discussion of realizability in the context of uncertainty quantification can be found in \cite{kusch2017maximum}.

By construction, the IPM system is hyperbolic, and, depending on the entropy chosen, the ansatz 
stays within certain physically relevant bounds \cite{poette2009uncertainty}.
These may be user-specified (through the choice of $s$) bounds which enforce 
the maximum-principle or the bounds of the domain of hyperbolicity (for 
example, for the Euler equations).
While for scalar problems the solution is non-oscillatory 
\cite{kusch2017maximum}, oscillations arise for systems like the 
Euler equations \cite{kusch2018filtered,kusch2020intrusive}, where the choice of admissible 
entropies is limited.
Furthermore, since the optimization problem \eqref{eq:primalProblem} can 
generally not be solved analytically, a numerical optimization method has to be 
used whenever the closure needs to be evaluated.
This heavily increases computational costs, though these costs should be able 
to be alleviated by the use of multi-elements \cite{kusch2020oscillation}, high-order methods for the temporal-spatial 
discretization as well as by taking advantage of parallel scalability 
\cite{garrett2015optimization}. 


The objective of this paper is to is to develop a method that suppresses oscillations in the solution while guaranteeing hyperbolicity of the moment system. This goal is achieved by incorporating a filter into the IPM closure. To extend the IPM method to filters, first we simply apply the filter to the numerical solution of \eqref{eq:momentEquationsClosed} (with the IPM closure) between time steps. 
But since an important property of the optimization problem is that it is 
infeasible when the given moment vector is not realizable, we must consider two 
options: either the filter must be chosen to preserve realizability or the optimization 
problem must be modified to become feasible for nonrealizable moment vectors. We propose two methods which follow these two paths. First, we design a filter which preserves realizability of the moments. Second, we change the IPM optimization problem to ensure a unique solution even if the moments are not realizable, allowing the use of common filters.

Parts of this work have been discussed in the PhD thesis \cite{10.5445/IR/1000121168}. To allow reproducing the numerical results of this work, the source code of the IPM solver implementation is publicly available \cite{uqcreator} together with scripts which recompute all test cases from this paper. The implementation is an extension of the code framework presented in \cite{kusch2020intrusive}.

Before proceeding, we note that there are other ways to enforce the physical 
bounds of the solution and dampen oscillations in the closure.
For example, one could enforce the physical bounds through constraints in the 
optimization problem instead of using the entropy $s$ in the objective function.
Also, instead of filtering the moment vector directly, oscillations could be 
dampened by penalizing them in the objective function.
These alternatives also likely merit investigation, but we leave this for 
future work.

The rest of this paper is structured as follows:
First we discuss how applying a standard filter to a realizable moment vector 
does not necessarily preserve its realizability and propose an alternate filter 
which does preserve realizability in Section~\ref{sec:realizability}.
Then in the following Section~\ref{sec:regularization}, we apply the 
regularization technique from \cite{alldredge2018regularized} to the 
optimization problem, thus making it feasible even for nonrealizable moment 
vectors and allowing the use of standard filters.
We discuss our implementation in Sections~\ref{sec:implementation} and present numerical results in 
Section~\ref{sec:results}. Section~\ref{sec:outlook} draws conclusions and gives an outlook on future work.

\section{Filtering and realizability}
\label{sec:realizability}

We consider two ways of incorporating a filter into the IPM method.
First, we show that the filtered moments $L(\lambda)\bm{\hat u}$ are not 
guaranteed to fulfill the realizability condition (even when $\bm{\hat u}$ 
does) and therefore propose a modified filter which does preserve 
realizability.
Second, we consider an alternative which modifies the IPM optimization 
problem \eqref{eq:primalProblem} instead of the filter so that the problem is 
feasible even for non realizable moment vectors.

\subsection{Effects of filtering on realizability}
Before presenting the two novel approaches to include filters in the IPM method, we demonstrate the core problem that comes with this task. I.e., we demonstrate that the filtered moments $L(\lambda) \bm{\hat{u}}$ are not necessarily realizable, even when the moments $\bm{\hat u}$ are.
Let us consider the scalar case, $m = 1$ and the entropy
\begin{align}\label{eq:kinEntropy}
 s(u) = u \log(u).
\end{align}
Now, a solution to the IPM optimization problem \eqref{eq:primalProblem} is only available if the moment vector lies in the realizable set $\mathcal{R}$, which we defined in \eqref{eq:realizableSetGeneral}. The realizable set which corresponds to the entropy \eqref{eq:kinEntropy} is
\begin{align*}
 \cR := \left\{ \bm{\hat u} \in \mathbb{R}^{N+1}
  \vert \exists u : \Theta \to (0, \infty) \text{ such that }
  \bm{\hat u} = \Vint{\bm{\varphi} u} \right\}.
\end{align*}
When $N = 2$ we have the more explicit form
\begin{align}
 \cR = \left\{ \bm{\hat u}\in \mathbb{R}^{N+1} \,\vert\, \hat u_0 > \hat u_2 > 0
  \text{ and } \hat u_0 \hat u_2 > \hat u_1^2 \right\}.
\end{align}
We refer to \cite{borwein1991convergence,shohat1943problem,curto1991recursiveness} for more information. In Figure~\ref{fig:realizabilityDomainsL2} we see that
$L(\lambda) \bm{\hat u}$ may be not realizable even when $\bm{\hat u}$ is.
In this figure, we plot
\begin{align*}
 L(\lambda) \cR |_{\hat u_0 = 1} = \{ L(\lambda)\bm{\hat u} \,|\,
  \bm{\hat u} \in \cR \text{ and } \hat u_0 = 1\}
\end{align*}
and see that for $\lambda \in \{ 0.05, 0.1, 0.2, 0.3 \}$ we have
${L(\lambda) \cR |_{\hat u_0 = 1} \not \subset \cR |_{\hat u_0 = 1}}$, i.e., 
there are realizable moments which become nonrealizable after the filter is 
applied.

\begin{figure}[h!]
\centering
\begin{subfigure}{.5\textwidth}
  \centering
  
\includegraphics[scale=0.55]{%
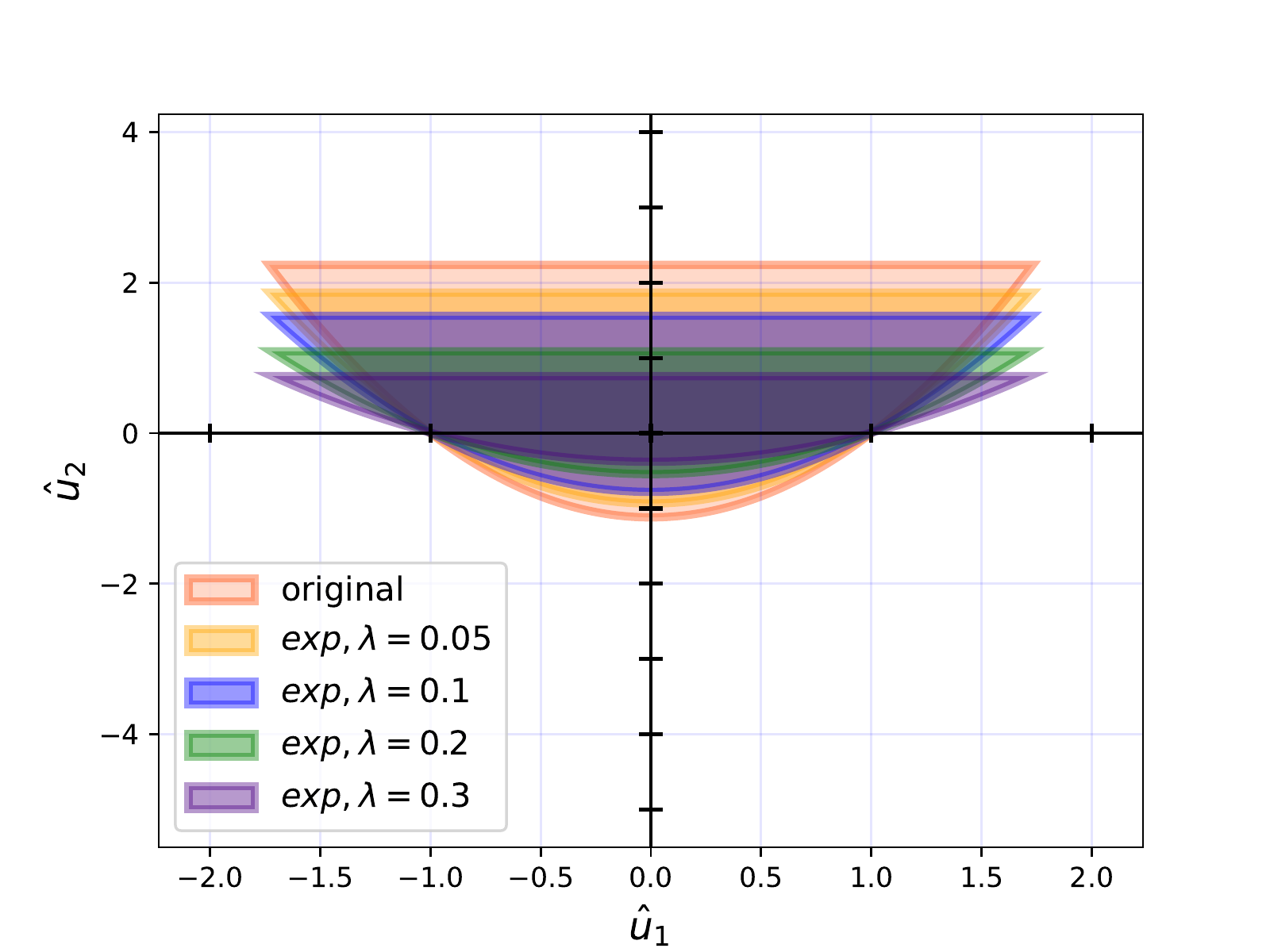}\hspace{-1.0cm}
  \caption{}
  \label{fig:realizabilityDomainsL2}
\end{subfigure}%
\begin{subfigure}{.5\textwidth}
  \centering
\includegraphics[scale=0.55]{%
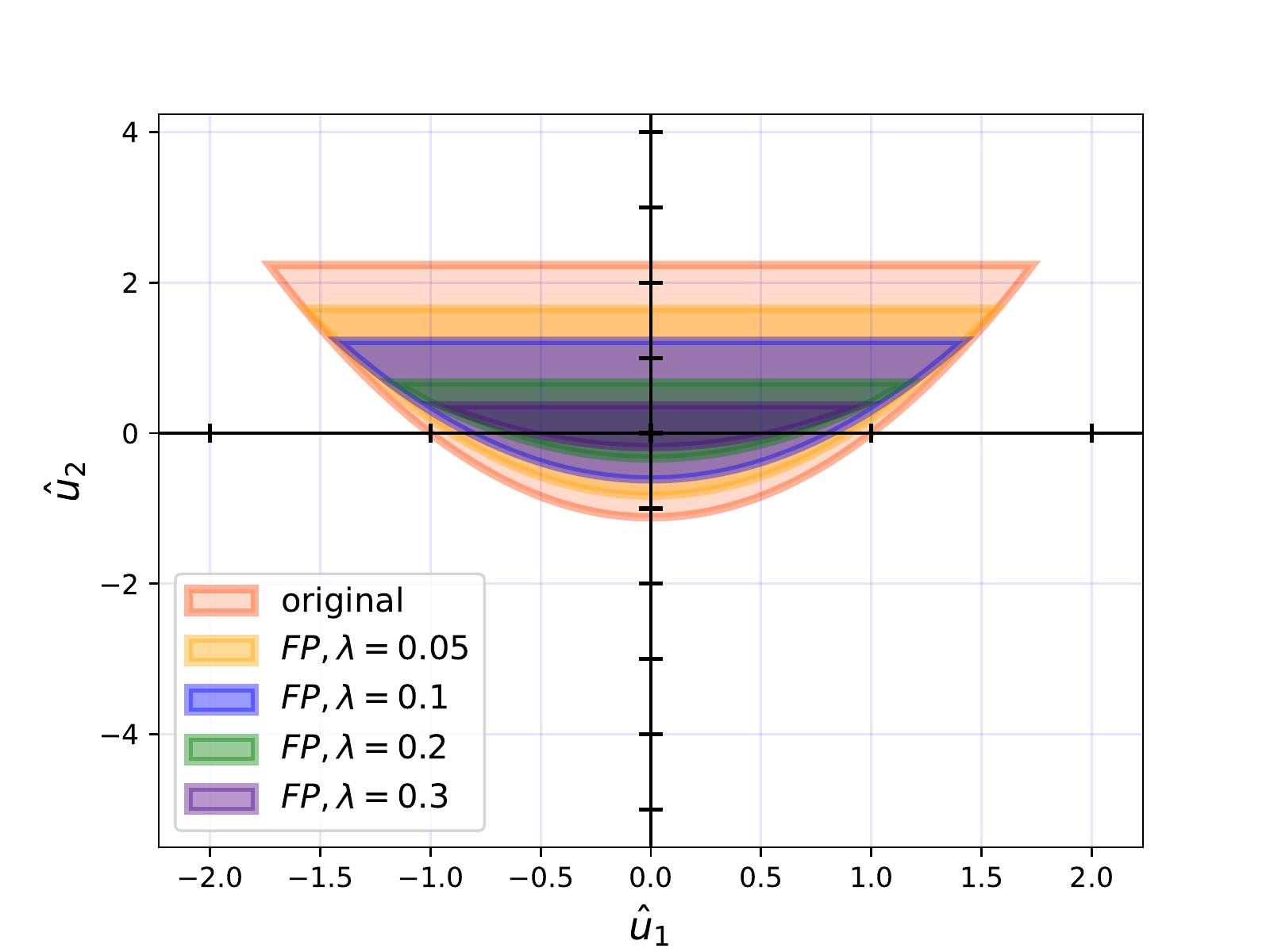}
 \vspace{0.2cm}
    \caption{}
     \label{fig:realizabilityDomainsFP}
\end{subfigure}
\caption{Images of the realizable set after application of filters. (a) Exponential 
filter with $\alpha = 7$ and $\Delta t = 0.1$, (b) Realizability-preserving Fokker-Planck filter}
\label{fig:realizabilityPlots2}
\end{figure}

On the other hand, in Figure~\ref{fig:realizabilityDomainsFP} we see the image 
of the realizable set $\cR$ after the application of a realizability-preserving 
filter, $F(\lambda)$, given below.
Here one sees that ${F(\lambda) \cR |_{\hat u_0 = 1} \subset \cR |_{\hat u_0 = 1}}$ for each  $\lambda$. In the following, we derive this realizability-preserving filter and discuss its properties.

\subsection{A realizability-preserving filter}
This section uses the Sturm-Liouville equation for orthogonal polynomials that we review in \ref{app:orthoPoly} for sake of completeness. Let us start by defining the filter for a scalar problem in the following theorem.
\begin{theorem}\label{th:RealizableFilterScalar}
Assume that the chosen gPC basis is the solution to a Sturm-Liouville equation. I.e., a self-adjoint eigenoperator $\mathcal{L}$ with eigenvalues $\mu_i$ for $i=0,\cdots,N$ exists. Furthermore, we assume that $\partial_t u = \mathcal{L}u$ fulfills a maximum principle. Let us define the filter function
\begin{align}\label{eq:FPFilter}
     g_{\lambda}(i) = \exp(\mu_i\lambda).
\end{align}
Then, making use of the corresponding filter matrix which takes the form $F(\lambda) = \diag\{ \exp(\mu_i\lambda) \}_{i = 0}^N$, the filtered moment vector
$F(\lambda)\bm{\hat u} = \vint{\bm{\varphi} u(\lambda, \cdot)}$ lies in the realizable set
\begin{align*}
 \cR := \left\{ \bm{\hat u} \in \mathbb{R}^{N+1}
  \vert \exists u : \Theta \to (u_-, u_+) \text{ such that }
  \bm{\hat u} = \Vint{\bm{\varphi} u} \right\},
\end{align*}
for scalar solutions (i.e., $m=1$) if the original moment vector $\bm{\hat u}$ is realizable.
\end{theorem}
\begin{proof}
The realizable moments, which we wish to filter can be written as $\bm{\hat u} = \vint{\bm{\varphi}u_{\text{ex}}}$, where $u_{\text{ex}}(\xi) \in (u_-, u_+)$. Consider the Fokker--Planck equation for $u = u(\lambda, \xi)$, where $\lambda$ plays the role of time:
\begin{align*}
 \partial_\lambda u = \cL u
  \qquad \text{ with } \qquad
 u(\lambda=0,\xi) = u_{\text{ex}}(\xi).
\end{align*}
The operator $\cL$ must be the eigenoperator (or Sturm-Liouville operator) to the chosen gPC basis. Since the solution to the Fokker--Planck equation fulfills the maximum principle and $u_{\text{ex}}(\xi) \in (u_-, u_+)$ for all $\xi \in \Theta$, we know $u(\lambda, \xi) \in (u_-, u_+)$ for all $\xi \in \Theta$ and $\lambda \ge 0$.
Therefore the moments
\begin{align}
 \Vint{\bm{\varphi} u(\lambda, \cdot)}
\end{align}
remain in the realizable set $\cR$ for all $\lambda \in [0, \infty)$. Furthermore, since $\cL$ is self-adjoint, these moments are straightforward to compute:
\begin{align}
 \partial_\lambda \Vint{\varphi_i u} = \Vint{\varphi_i \cL u}
  = \Vint{(\cL \varphi_i) u}
  = \mu_i \Vint{\varphi_i u},
\end{align}
and thus
\begin{align}
 \Vint{\varphi_i u(\lambda, \cdot)} =  \exp(\mu_i\lambda)
  \Vint{\varphi_iu_{\text{ex}}}.
\end{align}
Altogether with $\bm{\hat u} = \vint{\bm{\varphi}u_{\text{ex}}}$ and the definition of $F(\lambda) = \diag\{ \exp(\mu_i\lambda) \}_{i = 0}^N$, 
we can write the filtered moment vector as
$F(\lambda)\bm{\hat u} = \vint{\bm{\varphi} u(\lambda, \cdot)}$. The realizability of these moments is ensured by the 
maximum-principle satisfied by $u(\lambda, \xi)$.
\end{proof}
When using uniform distributions, the chosen gPC basis are Legendre polynomials for which all assumptions of Theorem~\ref{th:RealizableFilterScalar} hold. In this case, the eigenvalues are given by $\mu_i = -i(i+1)$, i.e. the filter matrix reads $F(\lambda) = \diag\{ \exp(-i(i+1)\lambda) \}_{i = 0}^N$. Note that since $\mu_0=0$, the filtered zeroth order moment $\vint{\varphi_0 u(\lambda, \cdot) }$ is constant 
in $\lambda$ and furthermore, as $\lambda \to \infty$, the solution
$u(\lambda, \xi)$ becomes constant in $\xi$.
Thus $\lambda$ indeed acts like a filter-strength parameter. 

Next we discuss how to extend this realizability-preserving filter to systems 
of conservation laws. First we define realizability more precisely for systems.
Let $\cR_{\bm{u}} \subseteq \R^m$ be the set of admissible states of the system.
For example, above in the scalar case we took $\cR_{\bm{u}} = (0, \infty)$ 
(although we could have just as easily taken
$\cR_{\bm{u}} = (u_{\min}, u_{\max})$, where $u_{\min}$ and $u_{\max}$ bound 
the initial condition, in order to enforce the maximum principle).
In the case of the Euler equations, one would take
\begin{align}
 \cR_{\bm{u}} = \left\{ (\rho, \rho \bm v, \rho E)
   \in \R^{d + 2} : \rho > 0, p > 0 \right\}
  = \left\{ (\rho, \rho \bm v, \rho E)
   \in \R^{d + 2} : \rho > 0, E > \frac12\Vert \bm v\Vert^2 \right\}
\end{align}
which is the set of $\bm{u}$ for which the Euler equations are hyperbolic. Here, $\rho\in\mathbb{R}$ denotes the density of the gas, $\bm{v}\in\mathbb{R}^d$ denotes the macroscopic gas velocity, $p\in\mathbb{R}$ is pressure and $ \rho E\in\mathbb{R}$ is total energy.
For a solution state $\bm{u}$ which depends on the uncertainty $\xi$, let
$\uhat = (\uhat_0, \uhat_1, \dots , \uhat_N)$ be the vector of moments with 
respect to the basis functions $\bm{\varphi}$.
We call $\uhat$ realizable when it belongs to the set
\begin{align}
 \cR := \left\{ \uhat \in \R^{(N + 1)\times m} : 
  \exists \bm{u} = \bm{u}(\xi) \in \cR_{\bm{u}}
  \text{ for all } \xi \in \Theta \text{ such that }
  \uhat = \Vint{\bm{\varphi} \bm{u}} \right\}.
\end{align}
We assume that the underlying admissible set has the form
\begin{align}
 \cR_{\bm{u}} := \left\{ \bm{u} \in \R^m : \exists f = f(v) \in I \subset \R
  \text{ for all } v \in V \subseteq \R^d \text{ such that }
  \bm{u} = \int_V \bm{\psi}(v) f(v) \, dv \right\}
\end{align}
where $\bm{\psi}$ is a given set of $m$ basis functions in $v$, and $I$ is 
usually $(0, \infty)$ or $(0, 1)$. For conservation laws derived from kinetic equations, $v$ plays the 
role of the velocity variable from the underlying kinetic equation.
This is the case for the Euler equations (where the $f$ is the Maxwellian,
$\bm{\psi}(v) = (1, v, \frac12|v|^2)$, and $I = (0, \infty)$) and more 
generally for entropy-based moment closures for kinetic equations 
\cite{levermore1996moment}.
Under this assumption, $\cR$ can be written as
\begin{align}\label{eq:Rmom}
 \cR = \Big\{ \uhat \in \R^{(N + 1)\times m} : \exists f =& f(\xi, v) \in I
   \text{ for all } (\xi, v) \in \Theta \times V \vphantom{\int_V} \nonumber\\
  &\text{ such that } 
  \uhat = \Vint{ \bm{\varphi}\int_V \bm{\psi}(v)^T f(\cdot, v) \, dv}^T
  \Big\}.
\end{align}

Now, since we have defined realizability for systems, let us discuss how the derived filter behaves in this case:
\begin{theorem}
Assume that all assumptions from Theorem~\ref{th:RealizableFilterScalar} hold. Applying the filter $F(\lambda)$ from Theorem~\ref{th:RealizableFilterScalar} component-wise for 
systems of conservation laws preserves realizability. Hence, if $\uhat_i\in\mathbb{R}^m$ collects the moments of order $i$ for all states $k = 1,\cdots,m$, the filtered moments $e^{\mu_i\lambda} \uhat_i$ for $i=0,\cdots,N$ remain realizable.
\end{theorem}
\begin{proof}
To filter a $\uhat \in \cR$ without destroying its realizability, let $f$ 
be any distribution from \eqref{eq:Rmom} and then apply $\cL$ to its 
$\xi$-dependence to define $g$:
\begin{align}\label{eq:f-filtered}
 \partial_{\lambda} g(\lambda, \xi, v) = \cL g(\lambda, \xi, v),
  \qquad
 g(0, \xi, v) = f(\xi, v).
\end{align}
Again, the maximum principle for $\cL$ ensures $g(\lambda, \xi, v) \in I$ for all
${(\lambda, v, \xi) \in [0, \infty) \times V \times \Theta}$.
Thus
\begin{align}
 \uhat(\lambda) := \Vint{ \int_V \bm{\psi} g(\lambda, \cdot, v) \, dv \bm{\varphi}^T}^T
  \in \cR
\end{align}
for all $\lambda \in [0, \infty)$.
A practical expression for $\uhat$ is also straightforward to derive, since
\begin{subequations}
\begin{align*}
 \partial_{\lambda} \uhat_i &= \Vint{\varphi_i \int_V \bm{\psi}
   \partial_{\lambda} g(\lambda, \cdot, v) \, dv} \\
   &= \Vint{\varphi_i \int_V \bm{\psi} \cL g(\lambda, \cdot, v) \, dv} \\
   &= \Vint{\cL\varphi_i \int_V \bm{\psi} g(\lambda, \cdot, v) \, dv} \\
   &= \mu_i \Vint{\varphi_i
    \int_V \bm{\psi} g(\lambda, \cdot, v) \, dv} \\
   &= \mu_i \uhat_i,
\end{align*}
\end{subequations}
i.e., $\uhat_i(\lambda) = e^{\mu_i\lambda} \uhat_i$.
(This also confirms that $\uhat_i(\lambda)$ is independent of the choice of $f$ in
\eqref{eq:f-filtered}.)
Thus we see that we can simply apply the filter $F(\lambda)$ component-wise for 
systems of conservation laws without destroying realizability.
\end{proof}

Note that the Fokker--Planck filter function \eqref{eq:FPFilter} is closely related to the exponential filter \eqref{eq:expFilter} being applied at order $\alpha=2$. Unfortunately, in uncertainty quantification high order filters appear to be a desirable choice, since they commonly yield satisfactory variance approximations, whereas low order filters heavily dampen the variance \cite{kusch2020oscillation}. Therefore, we now turn to allowing the use of standard filters (with potentially high order) by enabling the dual problem to cope with non-realizable moment vectors.

\subsection{Regularization}
\label{sec:regularization}

When faced with the potential infeasibility of the filtered moments, instead of 
modifying the filter to preserve realizability, we can instead modify the 
optimization problem to guarantee feasibility even for nonrealizable moment 
vectors.

A modification which achieves this is the regularization proposed in
\cite{Decarreau-Hilhorst-Lemarichal-Navaza-1992,alldredge2018regularized}, in
which the equality constraints are exchanged with a penalty term in the 
objective function.
Specifically, the primal optimization problem \eqref{eq:primalProblem} defining 
the ansatz is replaced by
\begin{align}\label{eq:primalRegularizedProblem}
 \mathcal{U}_{\eta}(\bm{\hat u}) = \argmin_{\bm{u}}\left\{ \Vint{s(\bm{u})} 
  + \frac{1}{2\eta} \left\|\bm{\hat u} - \Vint{\bm{\varphi} \bm{u}} \right\|^2\right\},
\end{align}
where $\eta \in (0, \infty)$ is the regularization parameter and
\begin{align}
 \|\bm{\hat w}\|^2 = \sum_{i = 0}^N \sum_{k = 1}^m \hat w_{ik}^2.
\end{align}
This regularized problem is feasible for a larger set of moment vectors
$\bm{\hat u}$, in particular when $\Theta$ is compact, it is feasible for all 
moment vectors.
 
Just as with the equality-constrained problem, the dual problem to 
\eqref{eq:primalRegularizedProblem} is finite dimensional:
\begin{align}\label{eq:dualRegularized}
 \bm{\hat v}_{\eta}(\bm{\hat u})
  := \argmin_{\bm{\hat v} \in \mathbb{R}^{m \times (N + 1)}} 
  \left\{\Vint{s_*(\bm{\hat v}^T \bm\varphi)}
  - \bm{\hat v} \cdot \bm{\hat u} + \frac{\eta}{2} \Vert \bm{\hat v} \Vert^2\right\},
\end{align}
and the form of the primal minimizer is the same:
\begin{align}
  \mathcal{U}_{\eta}(\bm{\hat u})
   = u_{ME}\left(\bm{\hat v}_{\eta}(\bm{\hat u})^T \bm{\varphi}\right),
\end{align}
and the flux for the resulting system is
\begin{align}\label{eq:regularizedFlux}
 \bm{f}(\bm{\hat u}) = \Vint{\bm{\varphi}
  \bm{\hat f}(\mathcal{U}_{\eta}(\bm{\hat u}))}.
\end{align}
This gives a hyperbolic system of conservation laws, and, when
$\bm{\hat v}_\eta$ is defined for all $\bm{\hat u}$, we can implement
\eqref{eq:filteredEquations} directly with a numerical flux $\bm{F^*}$ based
on \eqref{eq:regularizedFlux}.

One remaining question is how to choose the regularization parameter $\eta$. A discussion of this issue can be found in \cite{alldredge2018regularized}, where the authors choose $\eta$ according to the Morozov discrepancy principle, when interpreting the numerical errors as noise. Hence, $\eta$ is chosen such that the effect of regularization does not dominate the numerical error.

Before discussing the space and time discretization as well as implementation of the IPM scheme, let us take a look at the Fokker--Planck and exponential filter functions, which are given in equations \eqref{eq:FPFilter} and \eqref{eq:expFilter}. To get a better impression of how both filters dampen moments, we plot both filter functions in Figure~\ref{fig:filterFunction}. The chosen filter strength is $\lambda=0.04$. The time step size $\Delta t$ for the exponential filter is chosen such that the exponential filter at order two resembles the Fokker--Planck filter. It can be seen that Fokker--Planck filtering (with a scaled filter strength) behaves like aa second order exponential filter. When increasing the order of the exponential filter, the moment order $i$ at which the filter starts to dampen gPC coefficients increases. Note that this will especially affect the variance approximation, which (for a polynomial solution ansatz) is given by a sum over the squares of moments higher than zero.

\begin{figure}[h!]
    \centering
    \includegraphics[width=\linewidth]{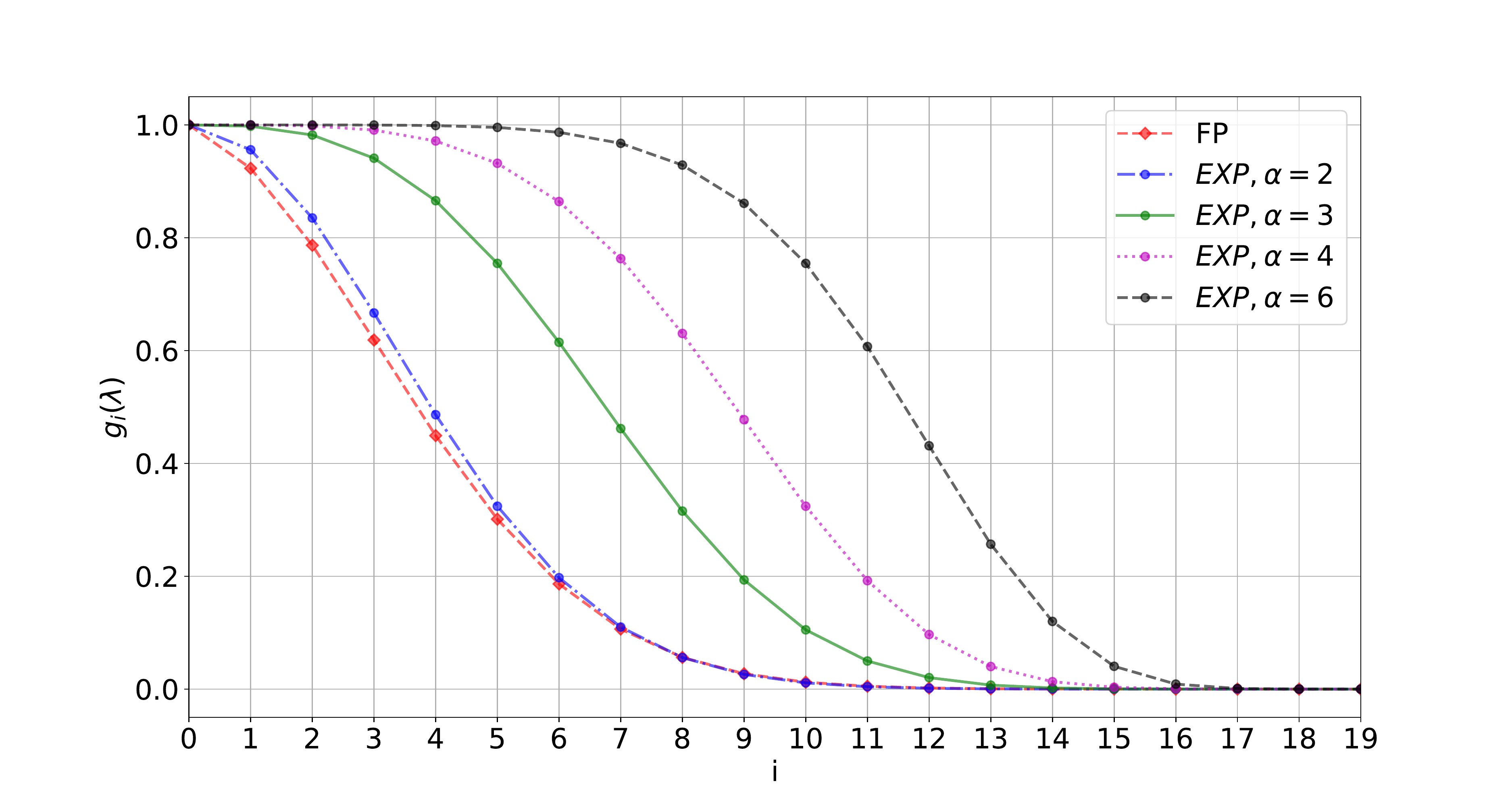}
    \caption{Fokker--Planck (FP) and exponential (EXP) filter functions for filter strength $\lambda=0.04$. The exponential filter uses a time step parameter of $\Delta t = 12.5$ and varying orders $\alpha$. The time step size is chosen such that the EXP filter for order $2$ resembles FP filtering.}
    \label{fig:filterFunction}
\end{figure}

\section{Implementation}\label{sec:implementation}

In this section we discuss the implementation of our numerical scheme for 
solving the filtered system \eqref{eq:filteredEquations}.


\subsection{Solving the dual problem}
The IPM method relies on solving the constrained optimization problem 
\eqref{eq:primalProblem}, or in the regularized case the unconstrained problem 
\eqref{eq:primalRegularizedProblem}.
Both of these problems are infinite-dimensional, so we solve their dual 
problems \eqref{eq:dualProblem} and \eqref{eq:dualRegularized}, which are 
finite-dimensional and unconstrained.
In both cases we used Newton's method stabilized with the standard backtracking 
linesearch.
Our stopping criterion is the norm of the dual, i.e.,
\begin{align}
 \|\Vint{\bm{\varphi} s'_*(\bm{\hat v}^T \bm{\varphi})} - \bm{\hat u} \| < \tau
  \qquad \text{or} \qquad
 \|\Vint{\bm{\varphi} s'_*(\bm{\hat v}^T \bm{\varphi})} + \eta \bm{\hat v}
  - \bm{\hat u} \| < \tau
\end{align}
for the original and regularized cases respectively, where
$\tau \in (0, \infty)$ is the user-specified tolerance parameter.
The integrals needed to compute the dual objective function and its derivatives 
are computed with quadrature.

The main challenge in the implementation of the IPM method (and minimal entropy 
methods in general) is that the optimization problem is ill conditioned when 
the moment vector lies near the boundary of the realizable set (or in the 
regularized case, when the regularization parameter is small and the moment 
vector of the solution lies close to the boundary of the realizable set).
In these cases the optimization algorithm needs many iterations to converge or 
may be unable to find a solution satisfying the stopping criterion. More information on this issue can for example be found in \cite{kusch2017maximum} for IPM and \cite{alldredge2012high} for minimal entropy methods in kinetic theory.

\subsection{Spatial and temporal dicsretization}\label{sec:implementationSpaceTime}
Let us now write down the implementation of the regularized IPM method 
including 
filters. In order to discretize the closed moment system 
\eqref{eq:momentEquationsClosed}, we choose a spatial grid $x_1,\cdots,x_{N_x}$ 
and a time grid $0=t_0,\cdots,t_{N_t}=t_{\text{end}}$. We then represent the 
dicretized solution by
\begin{align}\label{eq:discreteMoments}
\hat u_{ij}^n \simeq \frac{1}{\Delta x}\int_{x_{j-1/ 2}}^{x_{j+1/ 2}}\hat 
u_{i}(t_n,x) dx,
\end{align}
where $i$ is the moment order and $j$ is the spatial cell index.
The moment vector is collected in
${\bm{\hat u}_j^n = (\hat u_{ij}^n)_{i = 0}^N}$ for every 
spatial cell $j$ and time step $n$.
The corresponding dual state is denoted by
\begin{align*}
\bm{\hat v}_j^n := \bm{\hat v}(\bm{\hat u}_j^n),
\end{align*}
when reguarlization is not used; when regularization is used, $\bm{\hat v}$ is 
replaced with $\bm{\hat v}_\eta$
(see \eqref{eq:dualRegularized}).
Given a numerical flux $\bm{f^*}(\bm{u}_{\ell},\bm{u}_r)$ for the deterministic 
system, a flux for the moment system is constructed by
\begin{align}\label{eq:numFlux}
 \bm{F^*}(\bm{\hat u}(\bm{\hat v}_\ell), \bm{\hat u}(\bm{\hat v}_r)) := 
  \bm{G^*}(\bm{\hat v}_\ell, \bm{\hat v}_r) :=
  \left\langle \bm{\varphi} \bm{f^*}(s'_*((\bm{\hat v}_\ell)^T \bm{\varphi}),
  s'_*((\bm{\hat v}_r)^T \bm{\varphi})) \right\rangle.
\end{align}
We give this flux in terms of dual variables, since this more closely reflects the implementation. To compute the integrals in this expression one commonly uses a Gauss--Legendre quadrature rule, which equals the quadrature choice of the dual optimizer. This choice of the numerical flux is common in kinetic theory, where it is called the kinetic flux \cite{deshpande1986kinetic,harten1983upstream,perthame1990boltzmann,perthame1992second}. It is frequently used in the context of uncertainty quantification, see e.g. \cite{kusch2018filtered,kusch2020intrusive,kusch2017maximum,kusch2019adaptive,schlachter2018hyperbolicity,xiao2020stochastic}. 

The basic strategy of the algorithm, as given in \eqref{eq:filteredEquations},
is, in every time step, first to apply a filter to the moment vector in each 
spatial cell, and then perform an update in time starting from these filtered 
moment vectors.
Some details of the implementation vary depending on which method we choose, so 
we present the methods separately in Algorithms \ref{alg:fIPM-real} and
\ref{alg:fIPM-reg}.

In Algorithm \ref{alg:fIPM-real}, we note that after the dual variables 
are computed numerically for the filtered moment vector (using the 
Fokker--Planck filter), we compute in Line \ref{line:recompute} the moment 
vector associated with these approximte dual variables.
This is done to ensure the realizability of the moment vector at the next time 
step, $\bm{\hat u}_j^{n + 1}$, see \cite{kusch2017maximum}.
In Algorithm \ref{alg:fIPM-reg}, this step is unnecessary.

\begin{algorithm}[H]
\begin{algorithmic}[1]
\STATE $\bm{\hat u}^{0}_j \leftarrow \text{setup Initial Conditions}$ for all 
cells $j$
\STATE choose filter strength $\lambda \in (0, \infty)$
\FOR{$n=0$ to $N_t$}

\FOR{$j = 1$ to $N_x$}
\STATE $\bm{\overline{u}}_j^n \gets F(\lambda) \bm{\hat u}_j^n$
\STATE $\bm{\widetilde v}_j^n \gets \bm{\hat v}(\bm{\bar u}_{j}^n)$ 
using Newton's method with gradient tolerance $\tau$
\STATE $\bm{\widetilde{u}}_j^n \gets 
 \Vint{\bm{\varphi} s'_*((\bm{\widetilde v}_j^n)^T \bm{\varphi})}$  
 \label{line:recompute}
\ENDFOR
\FOR{$j=1$ to $N_x$}
\STATE $ \bm{\hat u}_j^{n+1} \leftarrow \bm{\widetilde{u}}_j^{n}
 - \frac{\Delta t}{\Delta x}(\bm{G^*}(\bm{\widetilde v}_j^n,
 \bm{\widetilde v}_{j+1}^n)
 - \bm{G^*}(\bm{\widetilde v}_{j-1}^n, \bm{\widetilde v}_j^n))$ 
\ENDFOR

\ENDFOR
\end{algorithmic}
\caption{Realizable Filtered IPM Method}
\label{alg:fIPM-real}
\end{algorithm}

\begin{algorithm}[H]
\begin{algorithmic}[1]
\STATE $\bm{\hat u}^{0}_j \leftarrow \text{setup Initial Conditions}$ for all 
cells $j$
\STATE choose filter strength $\lambda \in (0, \infty)$ and regularization 
parameter $\eta \in (0, \infty)$
\FOR{$n=0$ to $N_t$}

\FOR{$j = 1$ to $N_x$}
\STATE $\bm{\overline{u}}_j^n \gets L(\lambda) \bm{\hat u}_j^n$
\STATE $\bm{\widetilde v}_{\eta,j}^n \gets
 \bm{\hat v}_{\eta}(\bm{\bar u}_{j}^n)$ using Newton's method with gradient 
 tolerance $\tau$
\ENDFOR
\FOR{$j=1$ to $N_x$}
\STATE $ \bm{\hat u}_j^{n+1} \leftarrow \bm{\overline{u}}_j^{n}
 - \frac{\Delta t}{\Delta x}(\bm{G^*}(\bm{\widetilde v}_{\eta,j}^n,
 \bm{\widetilde v}_{\eta, j+1}^n)
 - \bm{G^*}(\bm{\widetilde v}_{\eta, j-1}^n, \bm{\widetilde v}_{\eta,j}^n))$
\ENDFOR

\ENDFOR
\end{algorithmic}
\caption{Regularized Filtered IPM Method}
\label{alg:fIPM-reg}
\end{algorithm}


\section{Results}
\label{sec:results}
In the following we will present numerical results for the two proposed filtering techniques. All results can be reproduced with the source code \cite{uqcreator}, which makes use of the intrusive UQ framework presented in \cite{kusch2020intrusive}.
\subsection{Effects of the Regularization}\label{sec:regulatizationRes}
\begin{figure}
\centering
\includegraphics[scale=0.4]{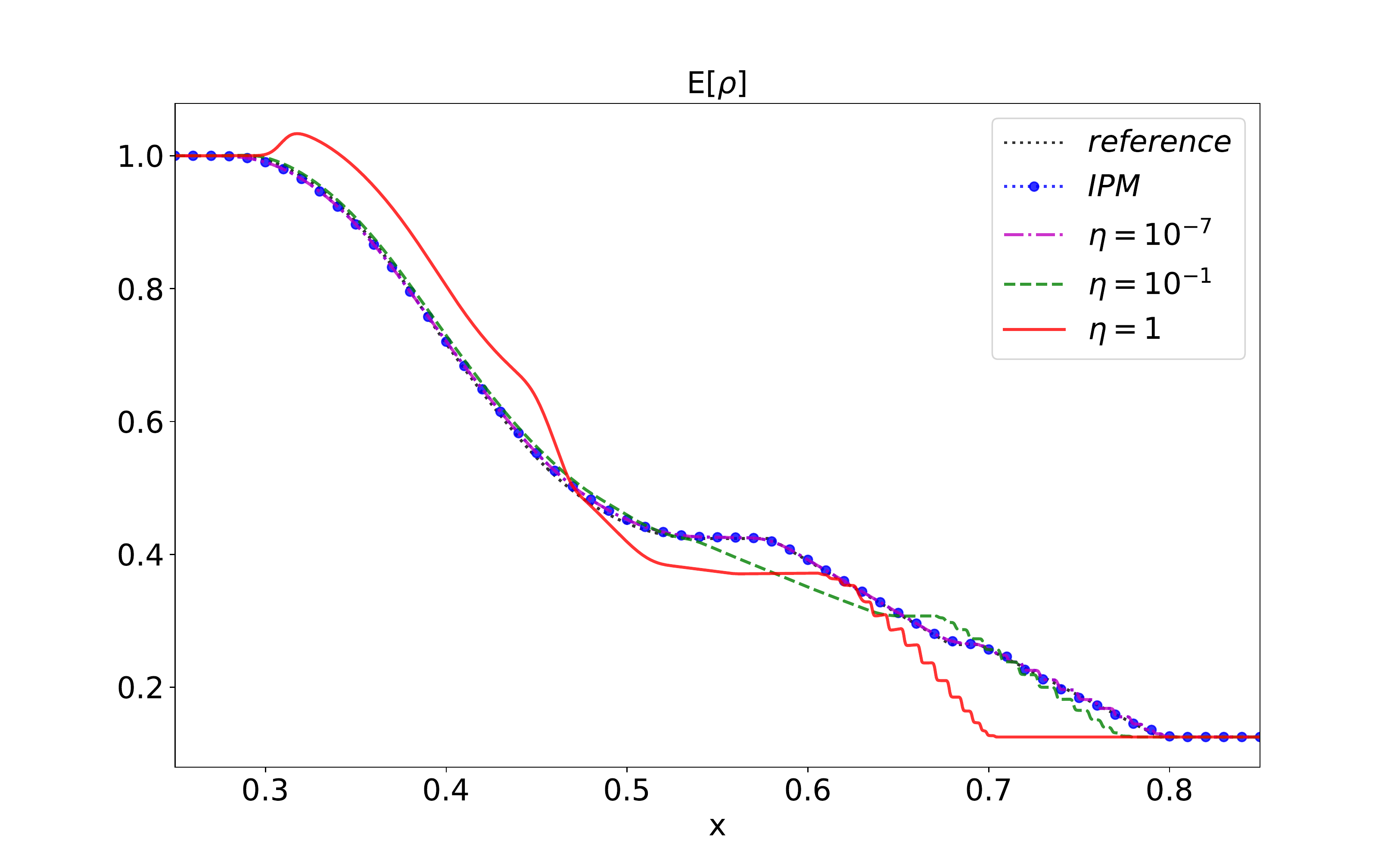}
\includegraphics[scale=0.4] 
{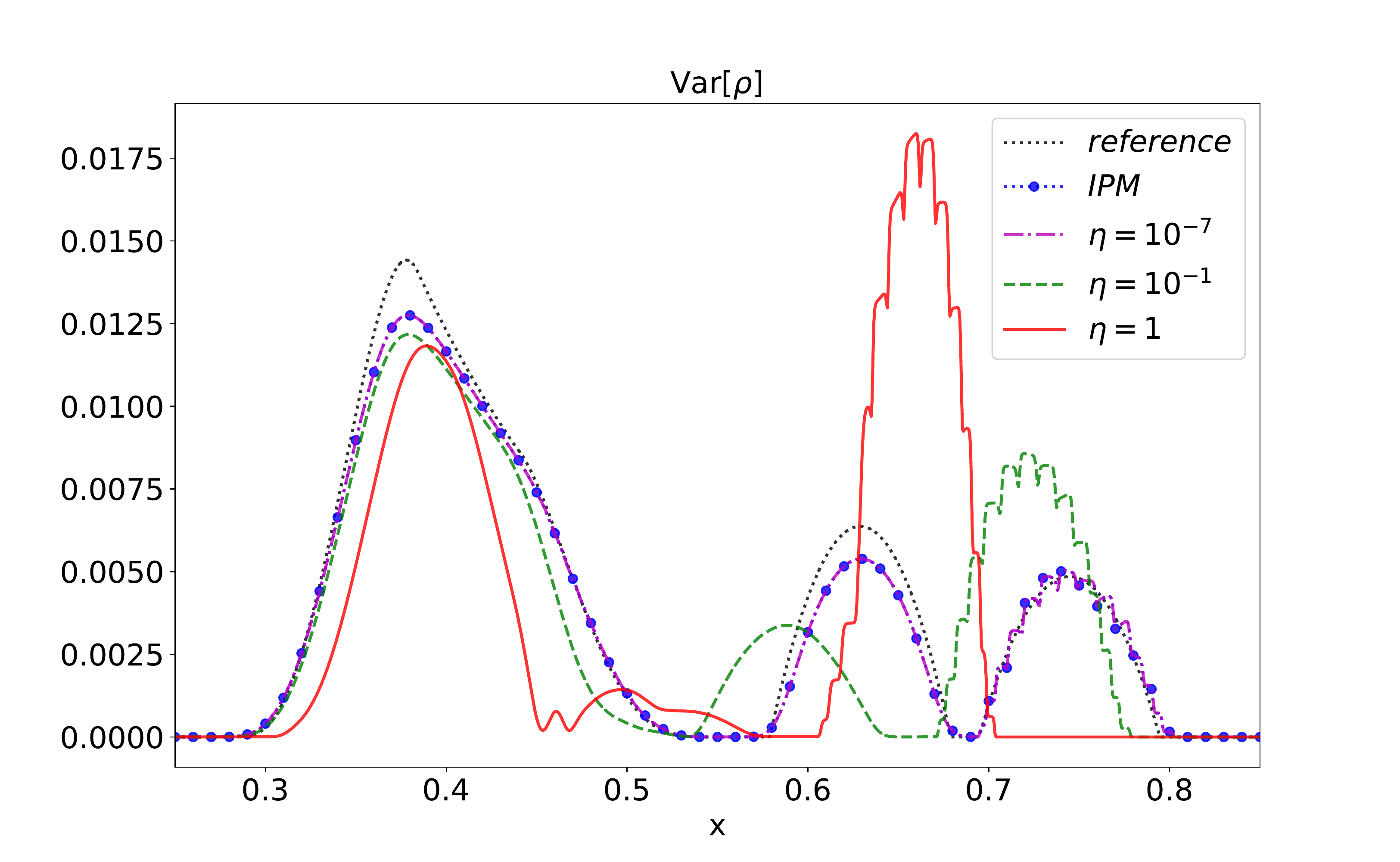}
\caption{Expected value and variance of gas density for different regularization strengths.}
\label{fig:regularizationEffects}
\end{figure}
Before moving to filtering, we study the effects of the regularization for 
different regularization strengths $\eta$. For this we investigate the 
following 
test case: We wish to solve the uncertain Sod shock tube problem, which is 
given 
by
\begin{align*}
\partial_t
\begin{pmatrix}
\rho \\ \rho v \\ \rho e
\end{pmatrix}
+\partial_x
\begin{pmatrix}
\rho v \\ \rho v^2 +p \\ v (\rho E+p)
\end{pmatrix}
=\bm{0},
\end{align*}
with the initial conditions
\begin{align*}
\rho_{\text{IC}} &= \begin{cases} \rho_L &\mbox{if } x < 
x_{\text{interface}}(\xi) \\
\rho_R & \mbox{else } \end{cases} \\
(\rho v)_{\text{IC}} &= 0 \\
(\rho E)_{\text{IC}} &= \begin{cases} \rho_L E_L &\mbox{if } x < 
x_{\text{interface}}(\xi) \\
\rho_R E_R & \mbox{else } \end{cases}
\end{align*}
Here, $\rho$ is the density, $u$ is the velocity and $e$ is the specific total 
energy. For a given heat capacity ratio $\gamma$ which in our case has a value 
of $1.4$, the pressure $p$ can be determined from
\begin{align*}
p = (\gamma-1)\rho\left(E-\frac{1}{2}v^2\right).
\end{align*}
The random interface position is given by $x_{\text{interface}}(\xi) = 
x_0+\sigma \xi$ with a uniformly distributed random variable $\xi\sim 
U([-1,1])$. We use Dirichlet boundary conditions at the left and right 
boundary. A reference solution is given by the analytic solution of the Sod shock tube test case. We approximate expected value and variance of the analytic solution by using a 100 point Gauss-Legendre quadrature rule.
The remaining parameter values are
\begin{center}
    \begin{tabular}{ | l | p{7cm} |}
    \hline
    $[a,b]=[0,1]$ & range of spatial domain \\
    $N_x=2000$ & number of spatial cells \\
    $t_\mathrm{end}=0.14$ & end time \\
    $x_0 = 0.5, \sigma = 0.05$ & interface position parameters\\
    $\rho_L,p_L = 1.0, \rho_R = 0.125, p_R = 0.1$ & initial states\\
    $N = 10$ & polynomial degree \\
    $N_q = 30$ & number of quadrature points \\
    $\tau = 10^{-7}$ & gradient tolerance for IPM \\
    \hline
    \end{tabular}
\end{center}
The IPM method uses the entropy
\begin{align*}
s(\rho,\rho v,\rho e) = \rho\ln{\left(\rho^{-\gamma}\left( \rho e - \frac{(\rho 
v)^2}{2 \rho} \right)\right)}.
\end{align*}
One observes that this problem cannot be solved with SG or fSG, since negative 
densities $\rho$ will show up already in the first time iteration. We now run 
this test case using IPM with different values for the regularization strength 
$\eta$. Hence, we run Algorithm~\ref{alg:fIPM-reg} without applying the filter. Taking a look at the 
resulting expected value and variance of the density $\rho$ in 
Figure~\ref{fig:regularizationEffects}, one sees that a big regularization 
parameter will heavily affect the solution. As the regularization strength 
decreases, the solution will approach the IPM solution and with $\eta =10^{-7}$ 
the regularized solution shows good agreement with IPM. When the regularization 
parameter is not sufficiently small, we observe significant effects on higher 
order moments: When looking at the variance of the rarefaction wave and the 
contact discontinuity, one notices strong dampening. Note however that the 
variance in these regimes is very sensitive to the method used, see for example 
\cite{kusch2018filtered,schlachter2018hyperbolicity}. It is important to point out that commonly, a regularization strength of $\eta = 10^{-7}$ is chosen and we compare against high values of the regularization strength to demonstrate the effects a regularization might have on the solution. For small regularization strength, the resulting numerical solution agrees well with the classical IPM solution.

The regularized solution with a small regularization parameter as well as the 
original IPM method show good approximation behavior in most parts of the 
solution. However, especially at the shock position, all methods will yield a 
non-physical, step-like approximation of the expected value and oscillatory 
results for the corresponding variance. Note that the step-like approximation of the IPM solution only becomes visible when further zooming into the shock region. For a deeper study of these artifacts we refer to \cite{kusch2020oscillation}. 

This work focuses on obtaining better 
approximations at the shock position through filtering techniques. Note that 
the 
filter will further dampen the variance approximations at the rarefaction wave 
and the shock discontinuity. One idea to mitigate these effects that we leave 
for future studies is choosing an adaptive filter strength as done in 
\cite{kusch2018filtered}.

\subsection{Filtering for Sod's shock tube}
\begin{figure}[h!]
\centering
	\begin{subfigure}{0.49\linewidth}
		\centering
		\includegraphics[scale=0.32]{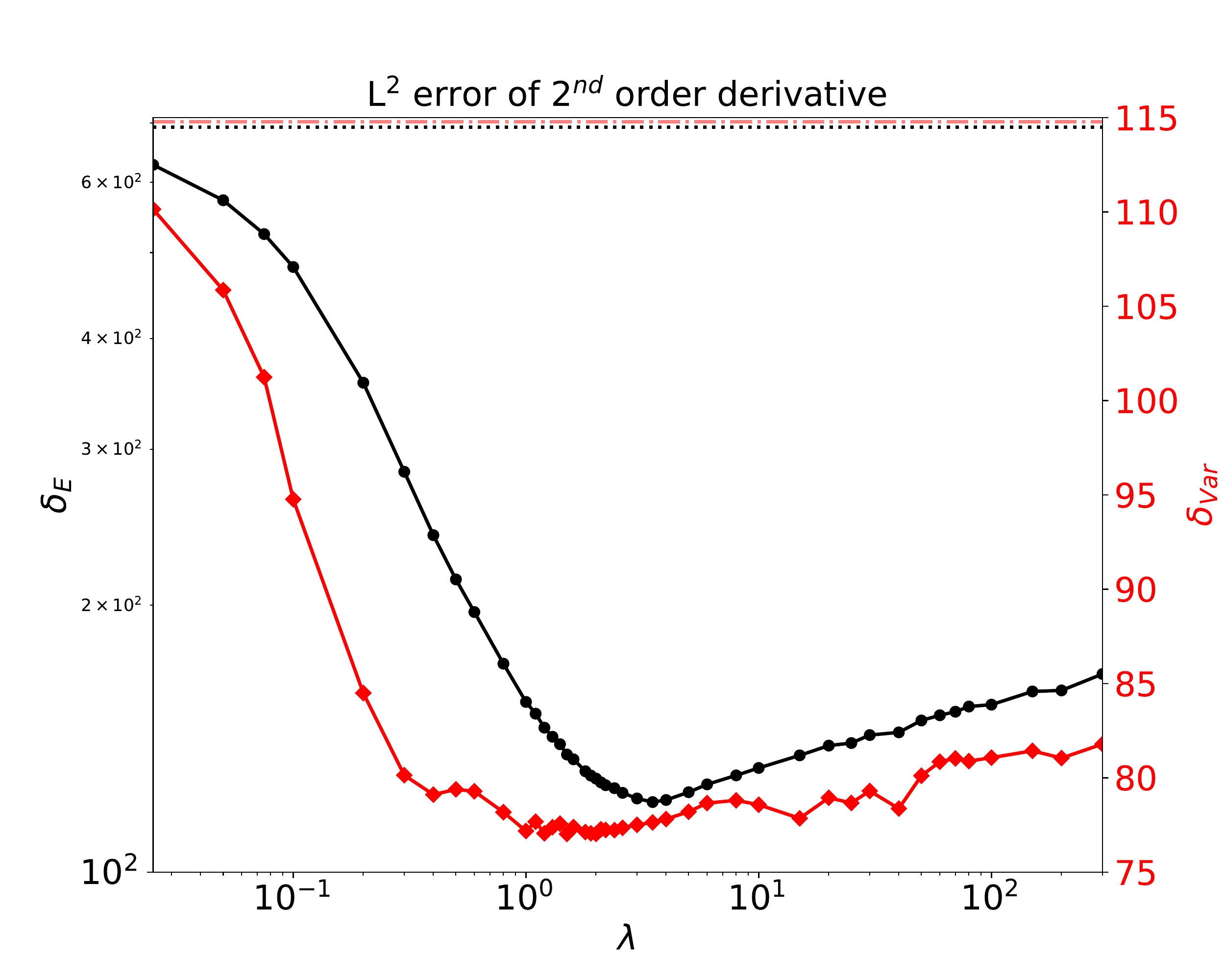}
		\caption{Exponential filter}
		\label{fig:ErrorsDeltasub1Sod}
	\end{subfigure}
	\begin{subfigure}{0.49\linewidth}
		\centering
		\includegraphics[scale=0.32]{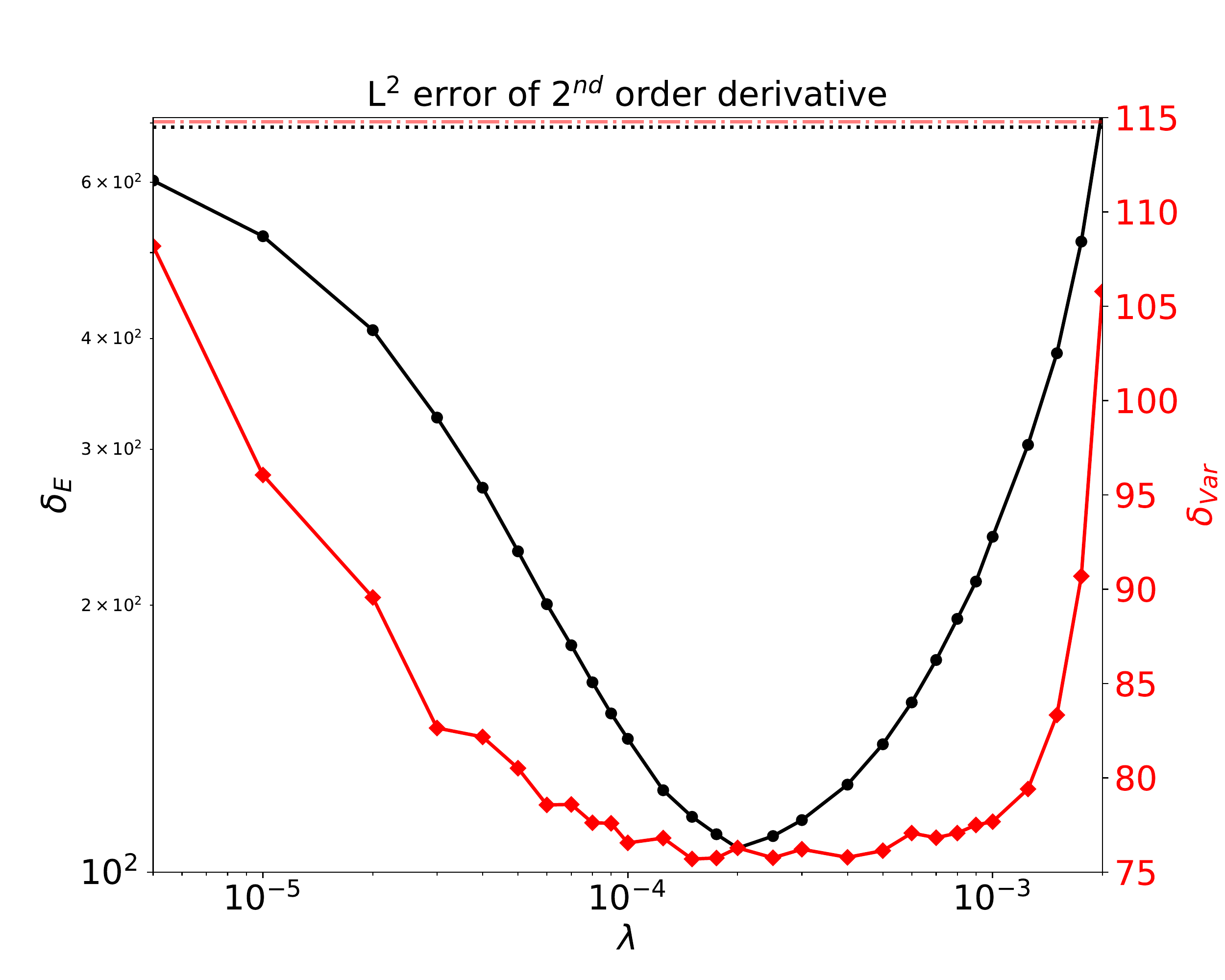}
		\caption{Fokker-Planck filter}
		\label{fig:ErrorsDeltasub2Sod}
	\end{subfigure}
	\caption{$\delta_{\text{E}}$ (in black with dots) and $\delta_{\text{Var}}$ (in red with diamonds) for different filter strengths when using the exponential and Fokker-Planck filter. The exponential filter uses a regularization strength $\eta=10^{-7}$ and order $\alpha=10$. Values for $\delta_{\text{E}}$ and $\delta_{\text{Var}}$ without filters are added as a straight line.}
	\label{fig:ErrorsDeltaSod}
\end{figure}

\begin{figure}[h!]
\centering
\begin{subfigure}{.5\textwidth}
  \centering
  
\includegraphics[scale=0.27]{%
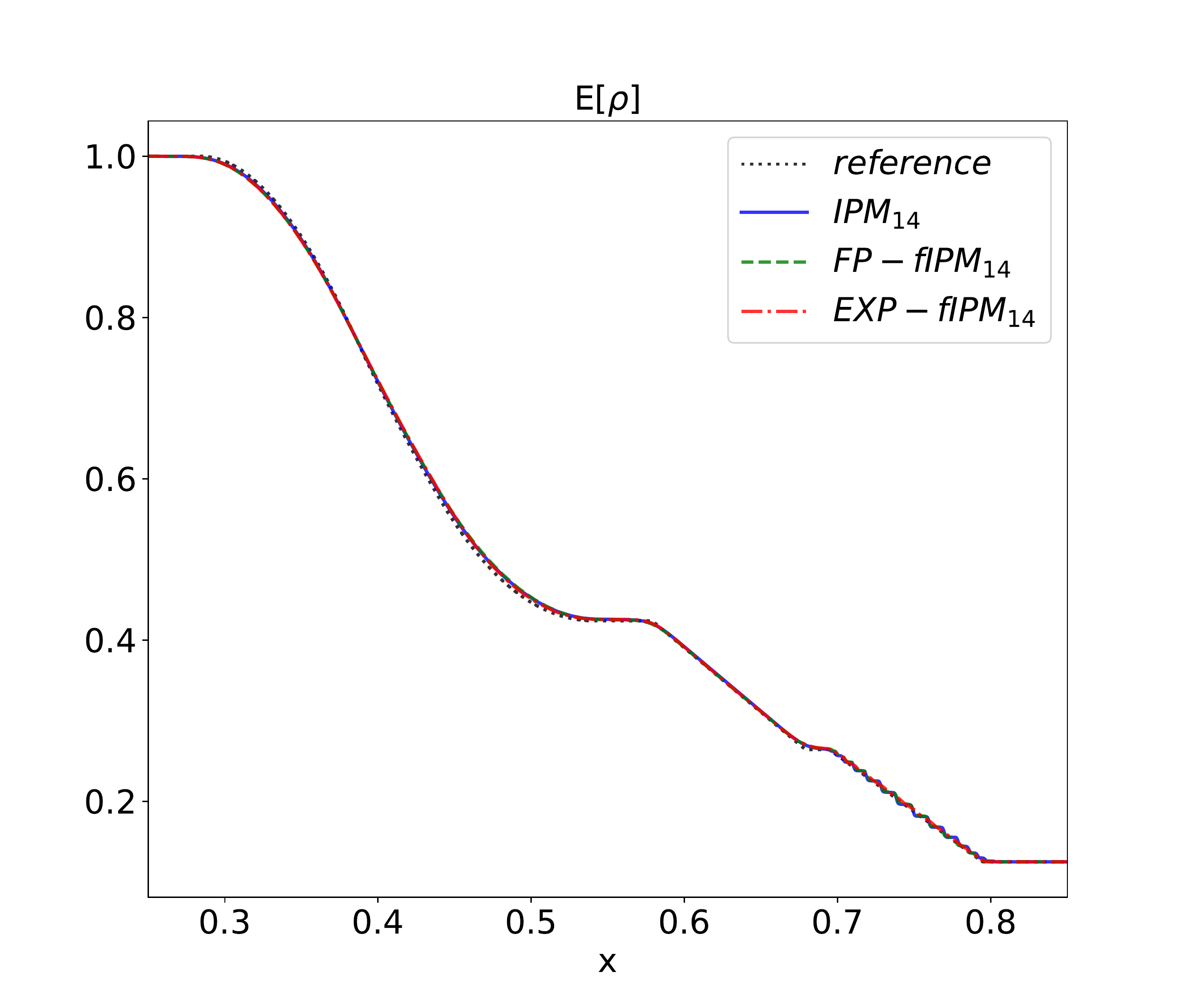}\hspace{-1.0cm}
  \caption{}
  \label{fig:CompareFilteringDensity}
\end{subfigure}%
\begin{subfigure}{.5\textwidth}
  \centering
\includegraphics[scale=0.27]{%
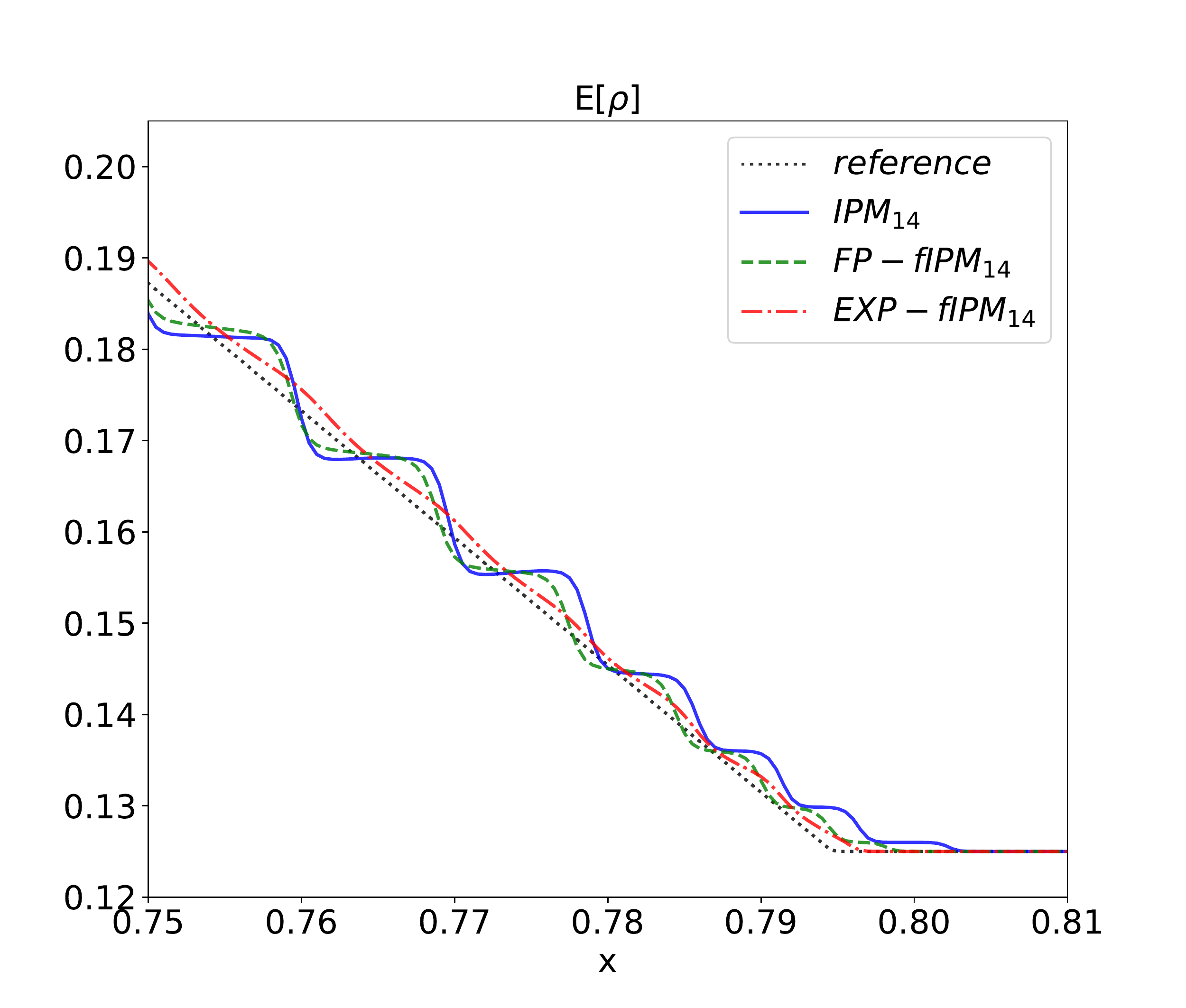}
 \vspace{0.2cm}
    \caption{}
     \label{fig:CompareFilteringZoomOnShock}
\end{subfigure}
\begin{subfigure}{.5\textwidth}
  \centering
  
\includegraphics[scale=0.27]{%
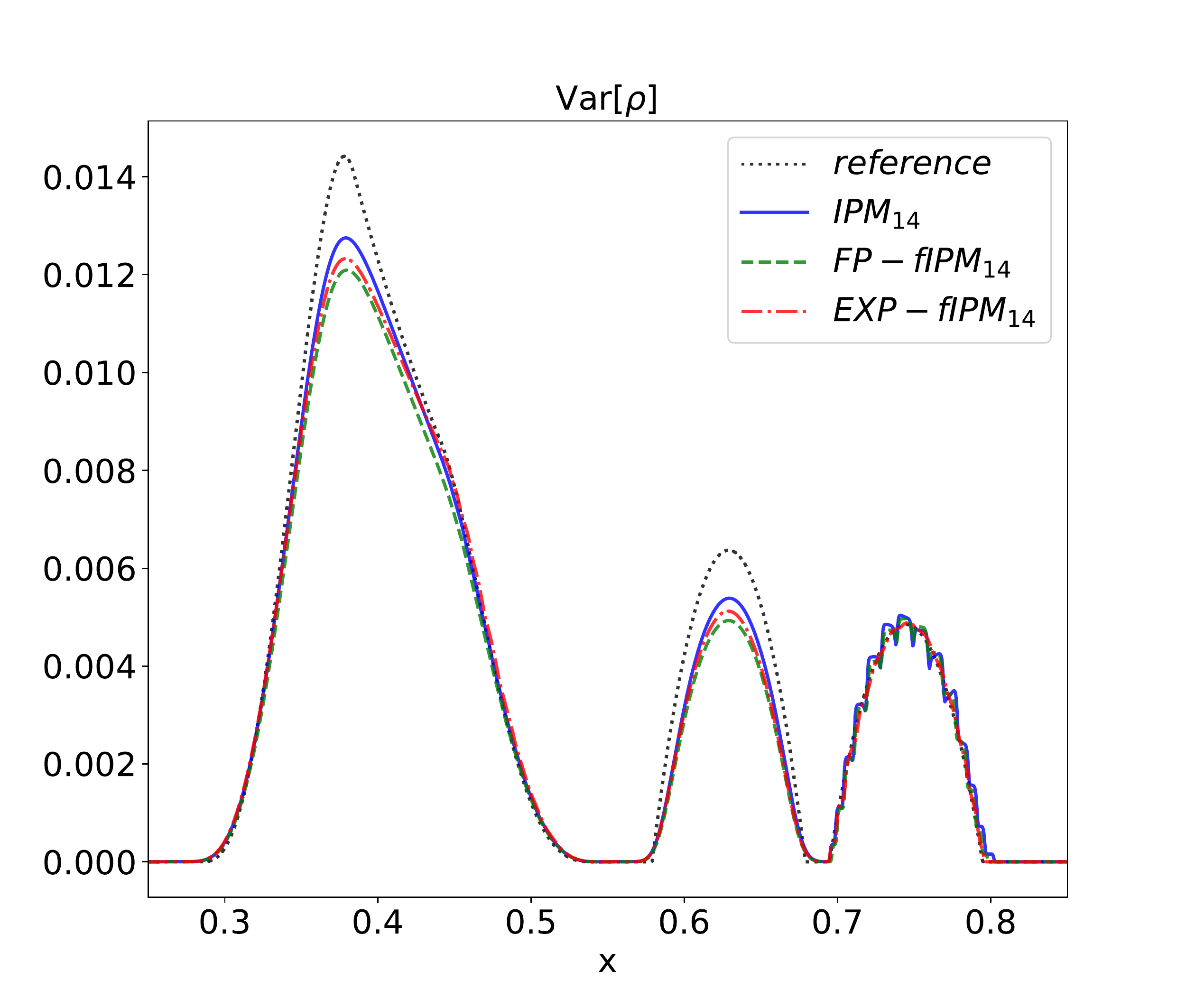}\hspace{-1.0cm}
  \caption{}
  \label{fig:CompareFilteringVar}
\end{subfigure}%
\begin{subfigure}{.5\textwidth}
  \centering
\includegraphics[scale=0.27]{%
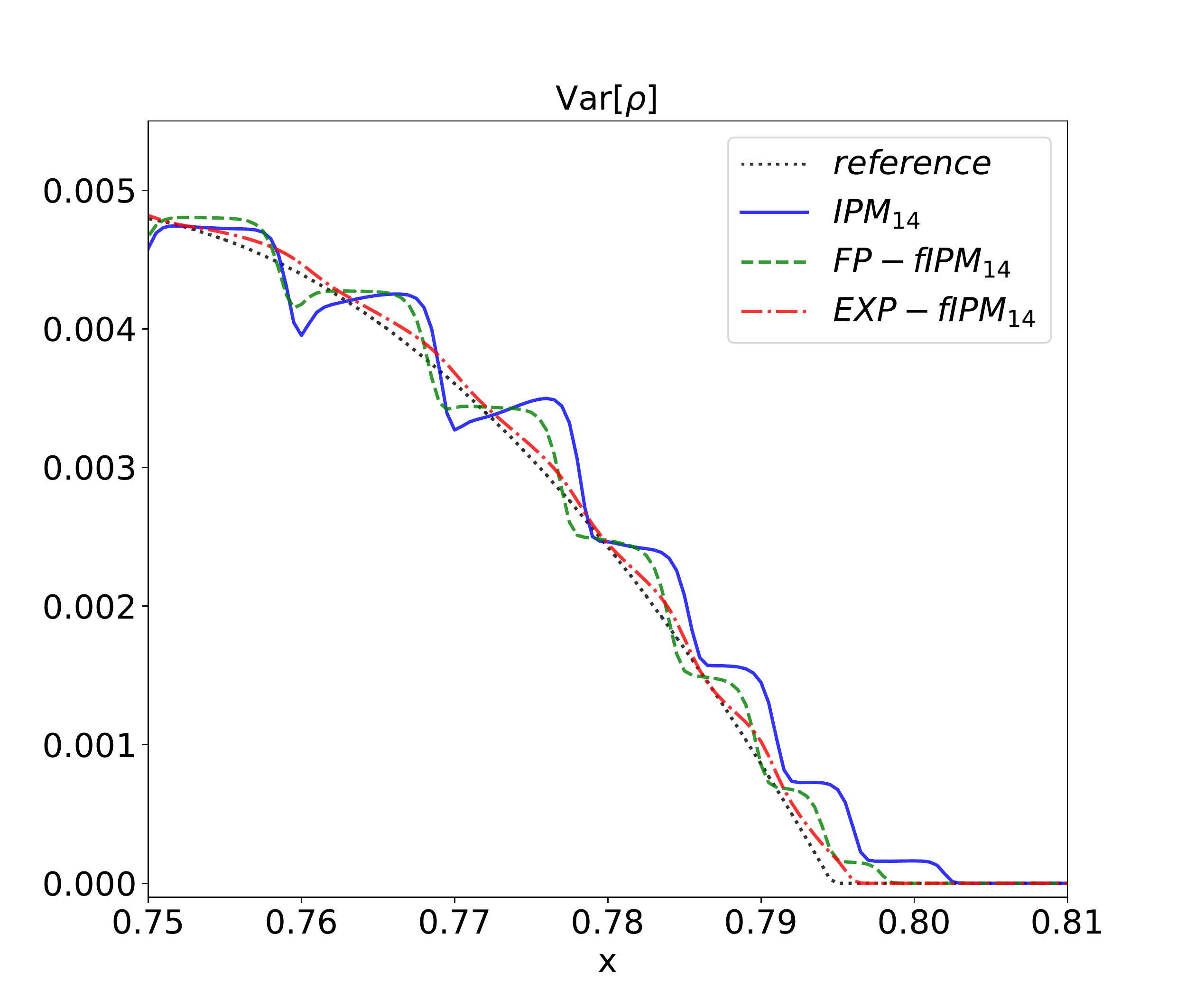}
 \vspace{0.2cm}
    \caption{}
     \label{fig:CompareFilteringVarZoomOnShock}
\end{subfigure}
\caption{Expectation and variance computed with IPM and filtered IPM ($\lambda=2,\alpha=10$ for exponential filtering and $\lambda=5\cdot10^{-5}$ for Fokker-Planck filtering) for Sod's shock tube. The exact expected value is depicted in red, the exact variance in blue. (a) Expected density. (b) Zoomed view on shock for expected density. (c) Variance of density. (d) Zoomed view on shock for variance of density.}
\label{fig:compareFilters}
\end{figure}

\begin{figure}[h!]
\centering
\begin{subfigure}{.5\textwidth}
  \centering
\includegraphics[scale=0.27]{%
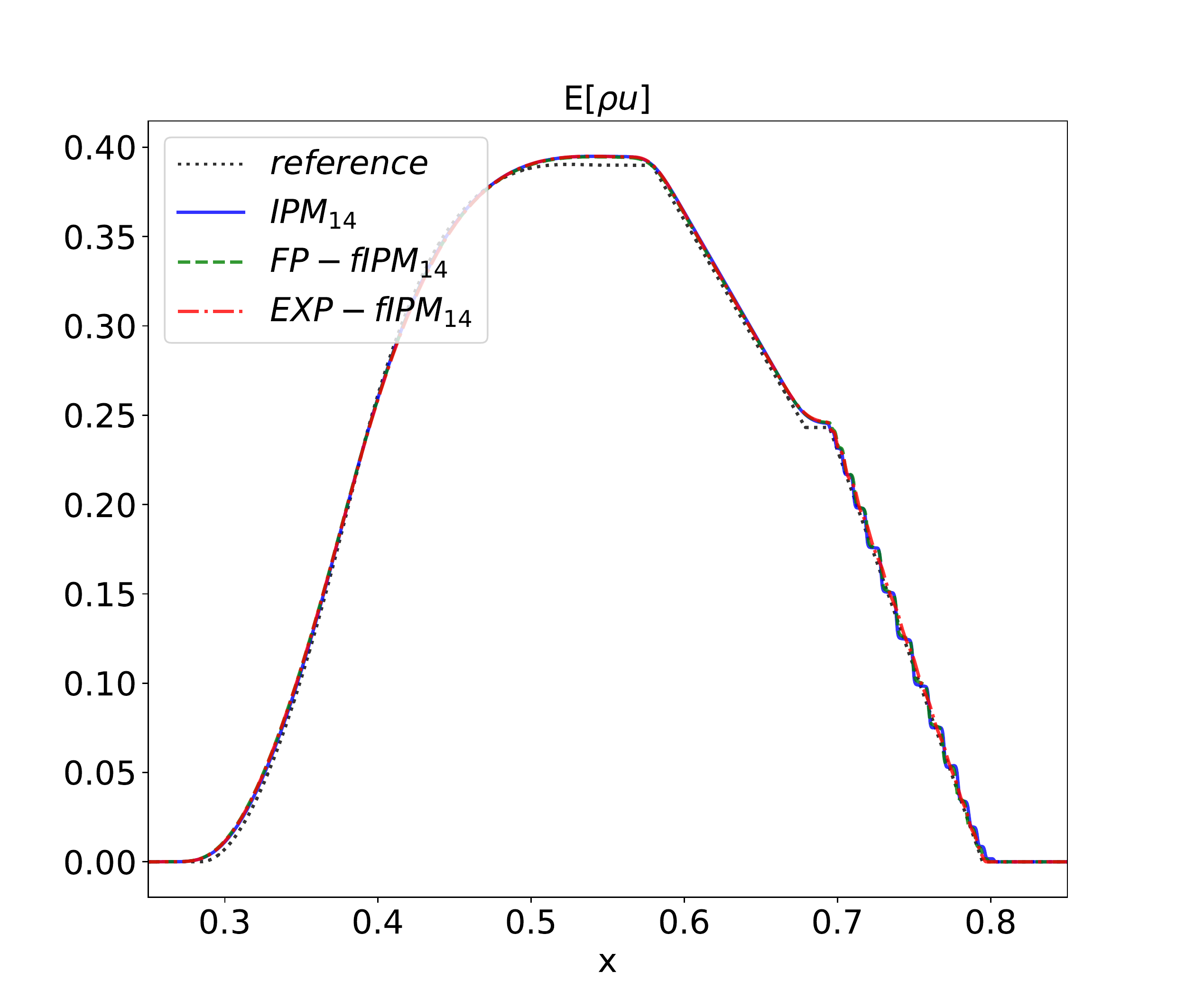}\hspace{-1.0cm}
  \caption{}
  \label{fig:CompareFilteringMomentum}
\end{subfigure}%
\begin{subfigure}{.5\textwidth}
  \centering
\includegraphics[scale=0.27]{%
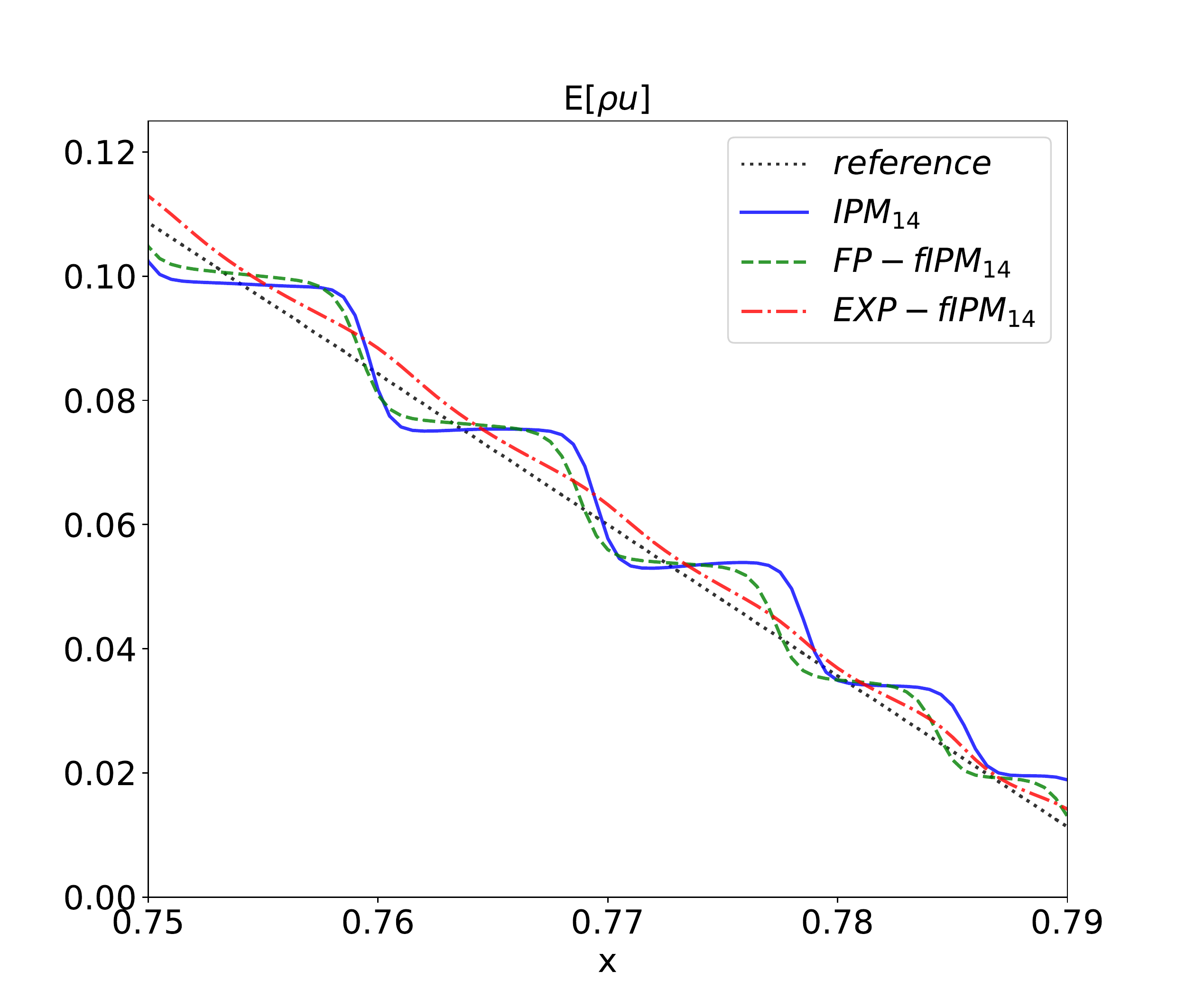}
 \vspace{0.2cm}
    \caption{}
     \label{fig:CompareFilteringMomentumZoomOnShock}
\end{subfigure}
\begin{subfigure}{.5\textwidth}
  \centering
\includegraphics[scale=0.27]{%
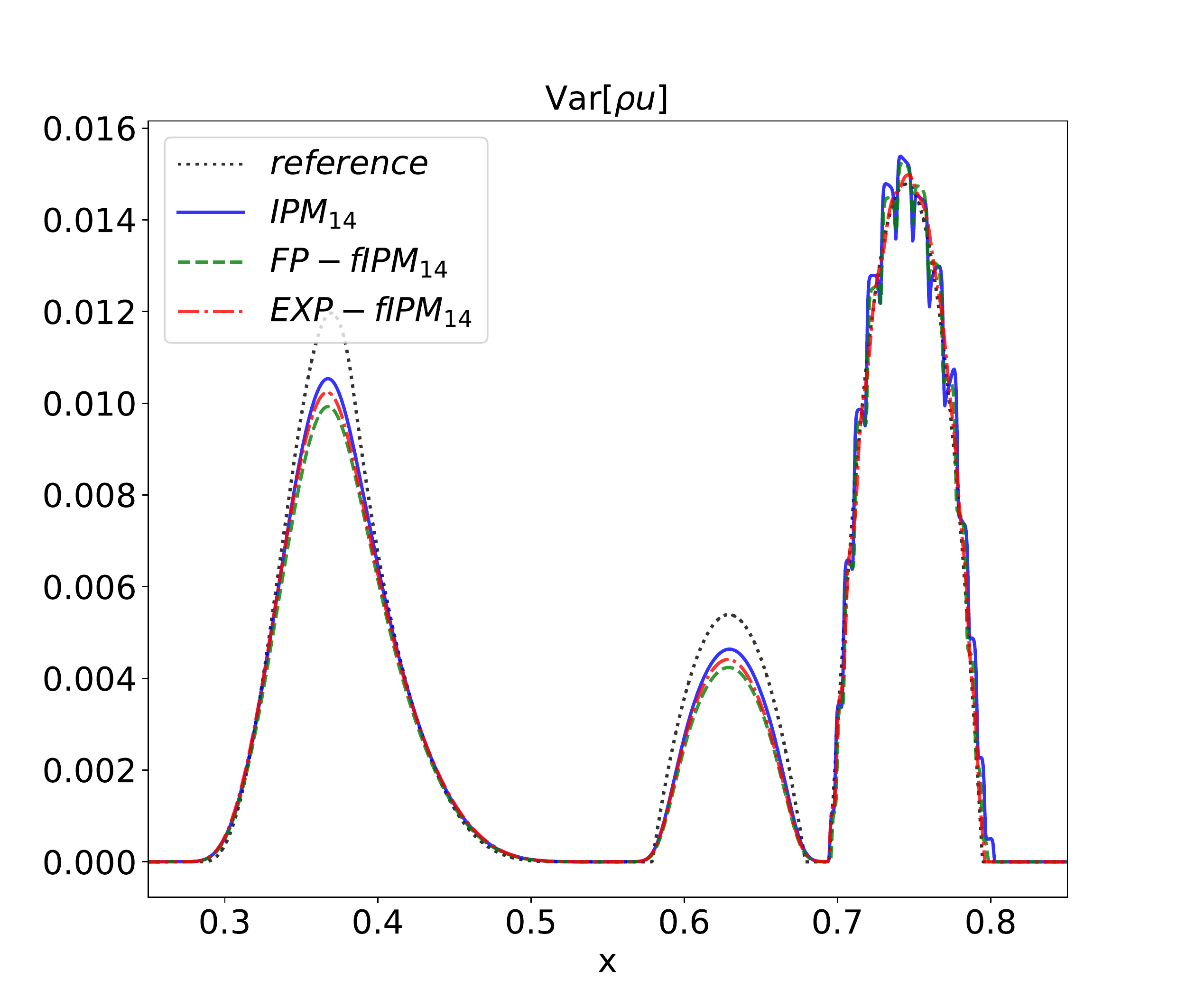}\hspace{-1.0cm}
  \caption{}
  \label{fig:CompareFilteringMomentum}
\end{subfigure}%
\begin{subfigure}{.5\textwidth}
  \centering
\includegraphics[scale=0.27]{%
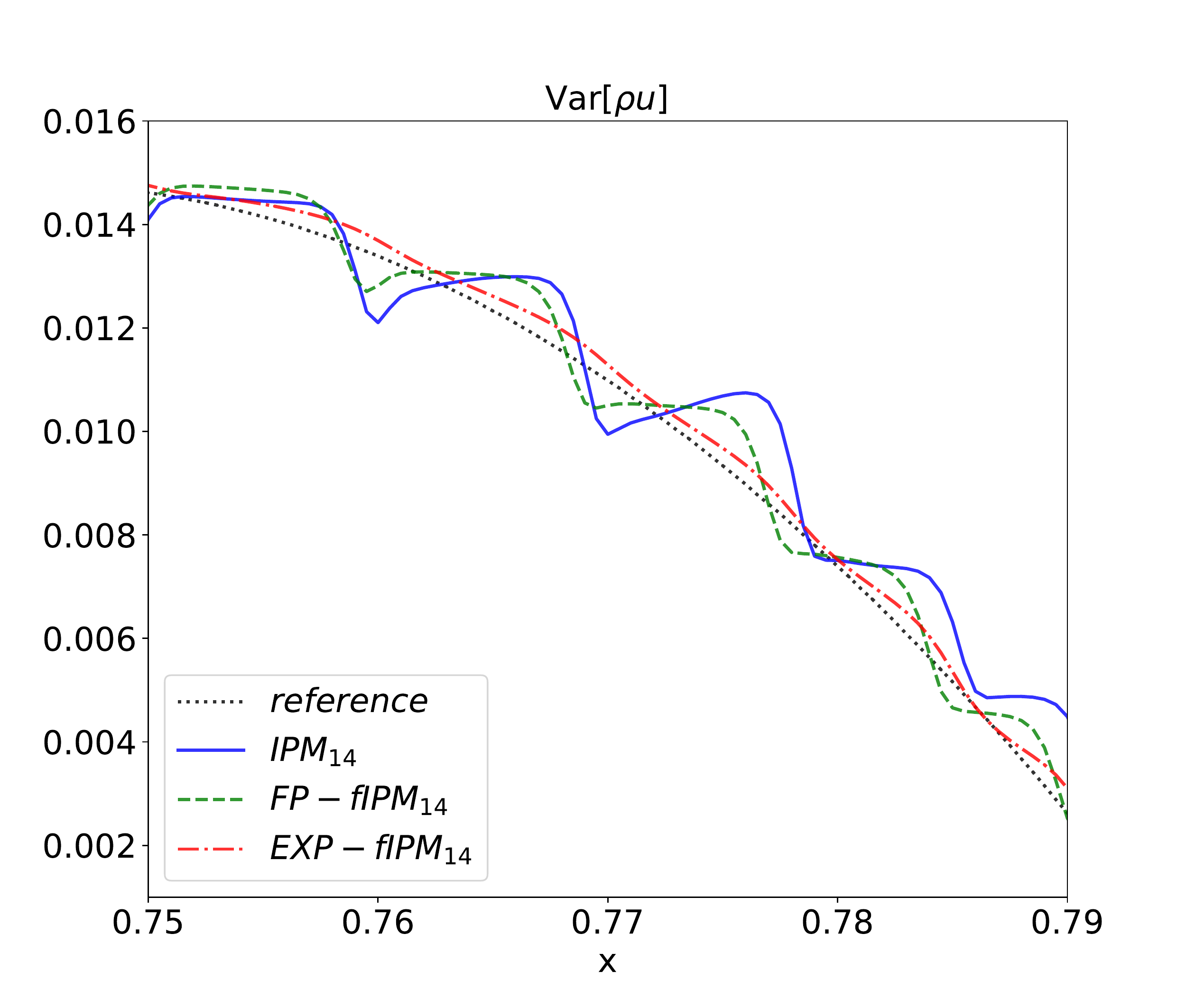}
 \vspace{0.2cm}
    \caption{}
     \label{fig:CompareFilteringMomentumZoomOnShock}
\end{subfigure}
\caption{Expectation and variance computed with IPM and filtered IPM 
($\lambda=2,\alpha=10$ for exponential filtering and $\lambda=5\cdot10^{-5}$ for 
Fokker-Planck filtering) for Sod's shock tube. The exact expected value is depicted in red, the 
exact variance in blue. (a) Expected momentum. (b) Zoomed view on shock for expected momentum. (c) Variance of momentum. (d) Zoomed view on shock for variance of momentum.}
\label{fig:compareFiltersMomentum}
\end{figure}
Let us start with studying the effects of the filter strength on the solution quality. As a measure for oscillations, let us investigate the error of second order derivatives. If the distance of a discrete numerical solution $u_{\Delta}:=(u_1,\cdots,u_{N_x})^T$ to a reference density $u_{\text{ex},\Delta}:=(u_{\text{ex},1},\cdots,u_{\text{ex},N_x})^T$ is given by $e_{\Delta} := u_{\text{ex},\Delta}-u_{\Delta}$, we wish to determine
\begin{align*}
\delta_{\text{E}}:=\sqrt{\sum_{j=1}^{N_x} \Delta x_j\left(  \partial_{xx} \mathbb{E}[e_j] \right)^2 }
\end{align*}
and
\begin{align*}
\delta_{\text{Var}}:=\sqrt{\sum_{j=1}^{N_x} \Delta x_j \left( \partial_{xx} \text{Var}[e_j] \right)^2 }.
\end{align*}
The resulting curves for both variables when choosing varying filter strengths can be found In Figure~\ref{fig:ErrorsDeltaSod}. Here and in subsequent uses of the exponential and erfc filter, we solve issues resulting from realizability by picking a constant regularization parameter $\eta=10^{-7}$. In all of our numerical experiments, this choice allowed solving the regularized dual problem. The filter order is here chosen to be $\alpha = 10$.  Let us point out that the chosen measures $\delta_{\text{E}}$ and $\delta_{\text{E}}$ do not necessarily provide a comparison between different methods and simply allows an idea of a meaningful range for the filter strength. Furthermore, we evaluate both error measures only in the shock region $x\in[0.7,0.8]$ to remove dampening effects from the rarefaction wave. It becomes clear that the optimal filter strength in terms of expectation and variance differs. Both filters reach an optimal variance value at a smaller filter strength compared to the optimal expectation. This means that while a bigger filter strength gives a good expectation approximation it heavily dampens out the variance, i.e., yielding a significant error in the uncertainty. Therefore, let us in the following choose $\lambda=2$ for the exponential filter. Fokker-Planck filtering seems to be optimal for values around $\lambda=10^{-4}$ around the shock. However, this choice leads to heavy dampening of the variance at the rarefaction wave, which is why we slightly reduce this value to $\lambda=5\cdot10^{-5}$. In this case, the Fokker-Planck filter shows a similar behaviour as the exponential filter at the rarefaction wave.

Note that we fixed the filter order $\alpha = 10$ for exponential filtering. Without going into detail, let us remark that increasing the order of the exponential filter mitigates the dampening effect at the rarefaction wave when increasing the filter strength. Thereby, a satisfactory value for $\lambda$, which improves the shock structure while preserving the remainder of the solution is easier to find. Choosing the order $\alpha$ too large will however require a significantly increased filter strength to obtain an improvement at the shock position. Since this behavior is difficult to capture by an error measure, we do not show a detailed parameter study as done for the filter strength. All remaining parameters are chosen as in section~\ref{sec:regulatizationRes}. Resulting expected values and variances are depicted in Figure~\ref{fig:compareFilters}. Again, the reference solution has been computed from the analytic solution of the Sod shock tube test case. As before, the expected value and variance of the analytic solution are approximated by a 100 point Gauss-Legendre quadrature rule.

We first examine the shock position: The expected value and variance of the density again show heavy oscillations in the variance and a step-like profile at the shock if no filtering is applied. Using the exponential filter leads to a smooth linear connection between the left and right shock state for the expected value, which shows good agreement with the exact solution. The variance is slightly dampened, however its smooth profile nicely captures the main characteristics of the exact variance. This behavior can be examined in detail when zooming onto the shock position as done in Figures~\ref{fig:CompareFilteringZoomOnShock} and \ref{fig:CompareFilteringVarZoomOnShock}. 
The Fokker-Planck filter appears to be less effective. Compared to standard IPM, the filter yields an improved approximation at the shock at the costs of dampening the variance at the rarefaction wave and contact discontinuity. When comparing exponential and Fokker-Planck filtering, one observes that exponential filtering yields a more satisfactory solution approximation, i.e., the shock region nicely captures the exact behavior while the variance at the rarefaction wave is only slightly damped. A similar conclusion can be drawn when looking at the expected momentum and its corresponding variance in Figure~\ref{fig:compareFiltersMomentum}.

There are different strategies that can be chosen to mitigate non-physical dampening effects at the rarefaction wave.
One example is using an adaptive filter strength. This approach can be combined with adaptivity in general \cite{kusch2020intrusive}. In 
this work, however, we focus on mitigating oscillations at the shock and leave 
other open questions to future work.

\subsection{Filtering for high density shock}\label{sec:resultsHighDensity1D}
In the following, we keep the same parameter setting, but change the initial condition to verify if the previous choice of the filtering parameters still yields a satisfactory result. The test case we investigate has an increased density state of $\rho_R=0.8$ on the right hand side. The remaining shock states are left untouched. The resulting expected density and corresponding variance can be found in Figure~\ref{fig:compareFiltersDensityHD}. 
\begin{figure}[h!]
\centering
\begin{subfigure}{.5\textwidth}
  \centering
\includegraphics[scale=0.27]{%
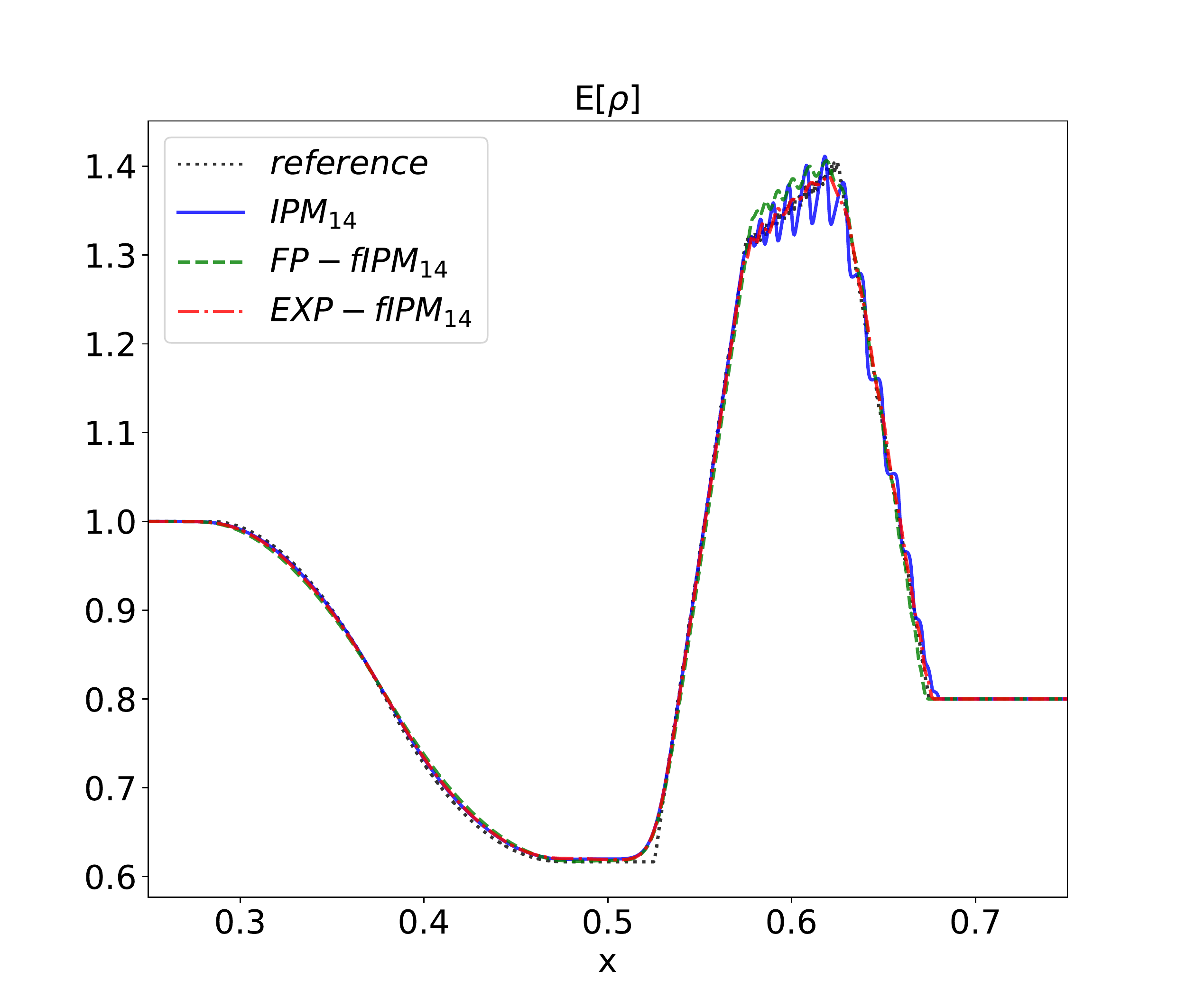}\hspace{-1.0cm}
  \caption{}
\end{subfigure}%
\begin{subfigure}{.5\textwidth}
  \centering
\includegraphics[scale=0.27]{%
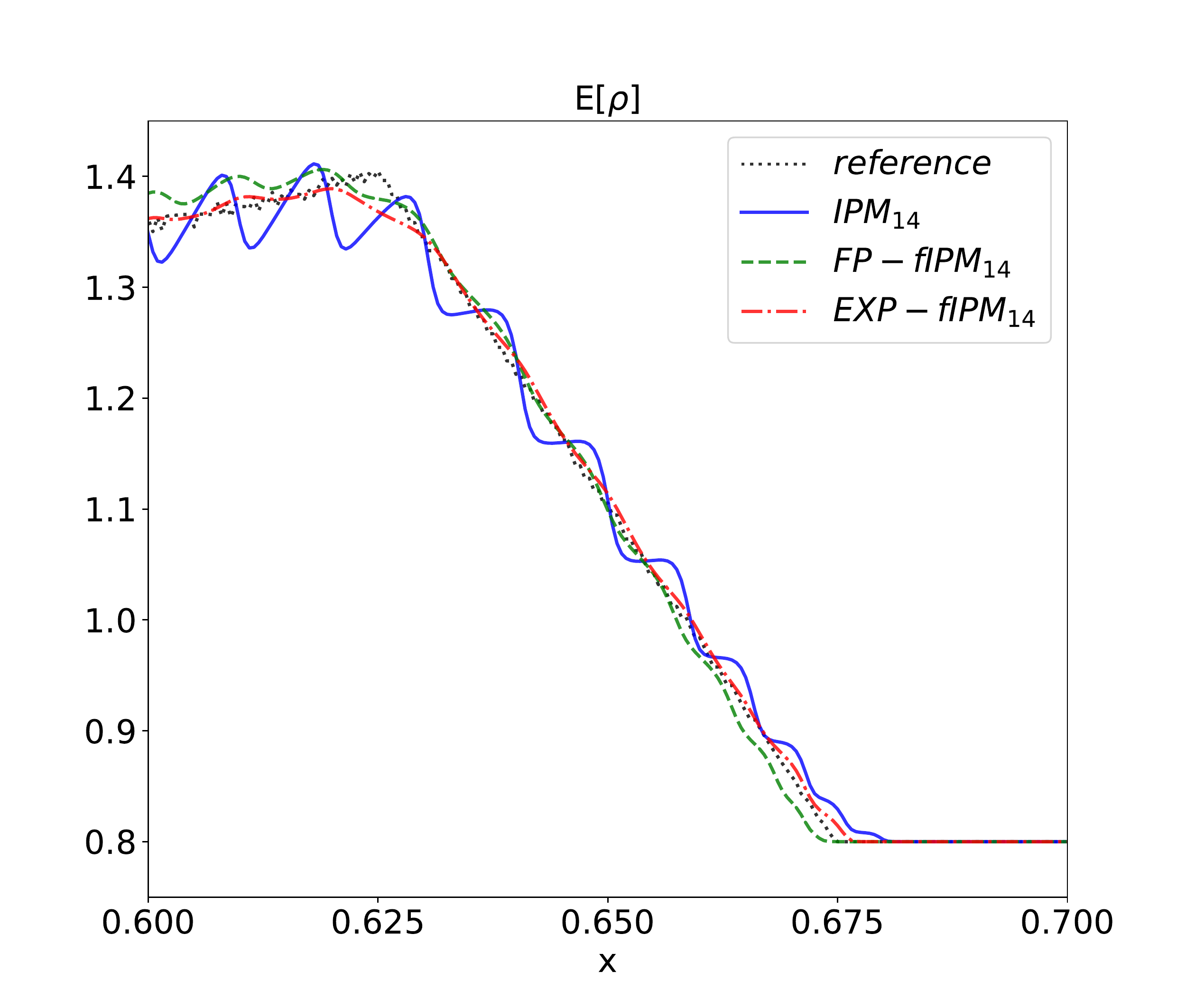}
 \vspace{0.2cm}
    \caption{}
\end{subfigure}
\begin{subfigure}{.5\textwidth}
  \centering
\includegraphics[scale=0.27]{%
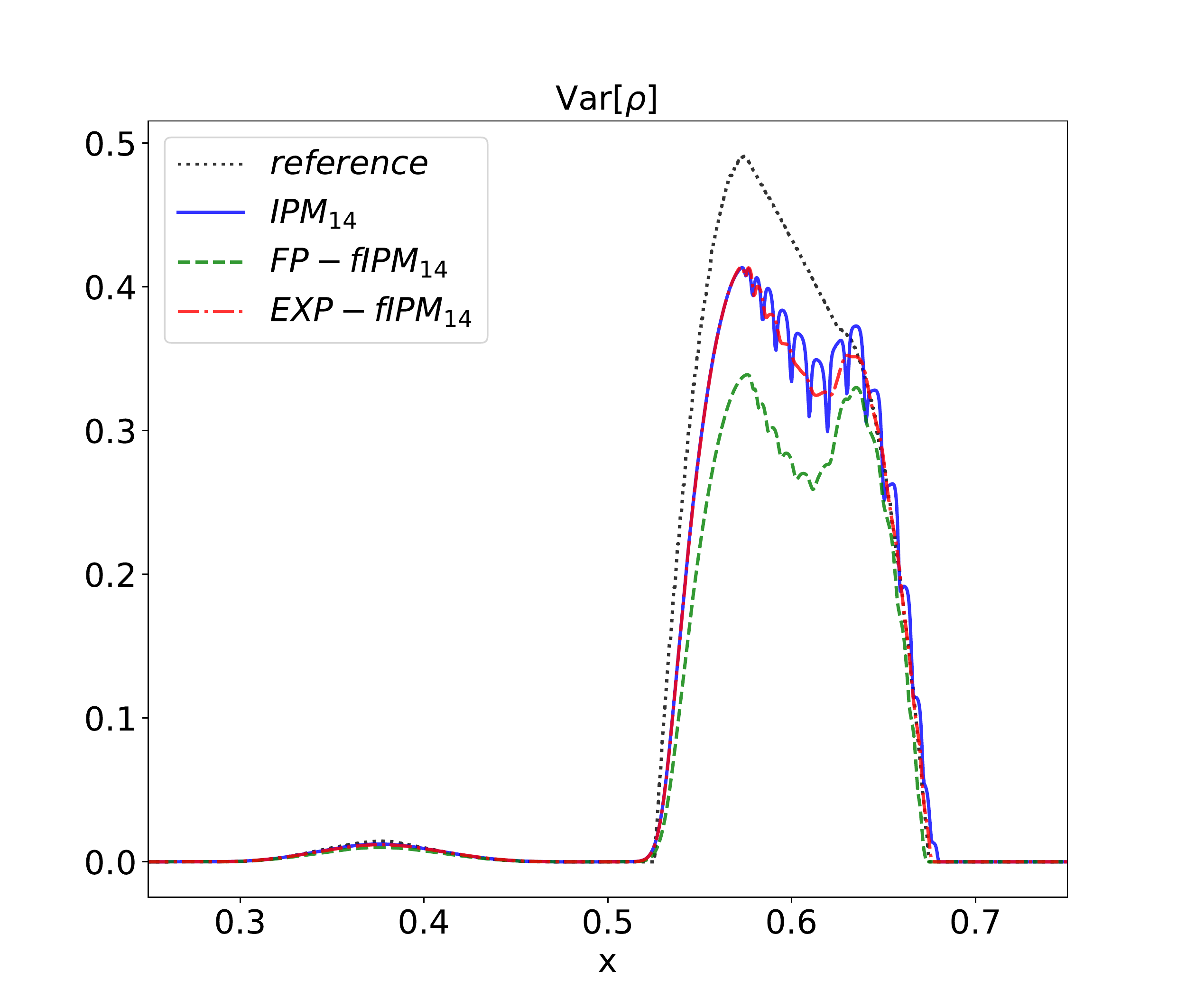}\hspace{-1.0cm}
  \caption{}
\end{subfigure}%
\begin{subfigure}{.5\textwidth}
  \centering
\includegraphics[scale=0.27]{%
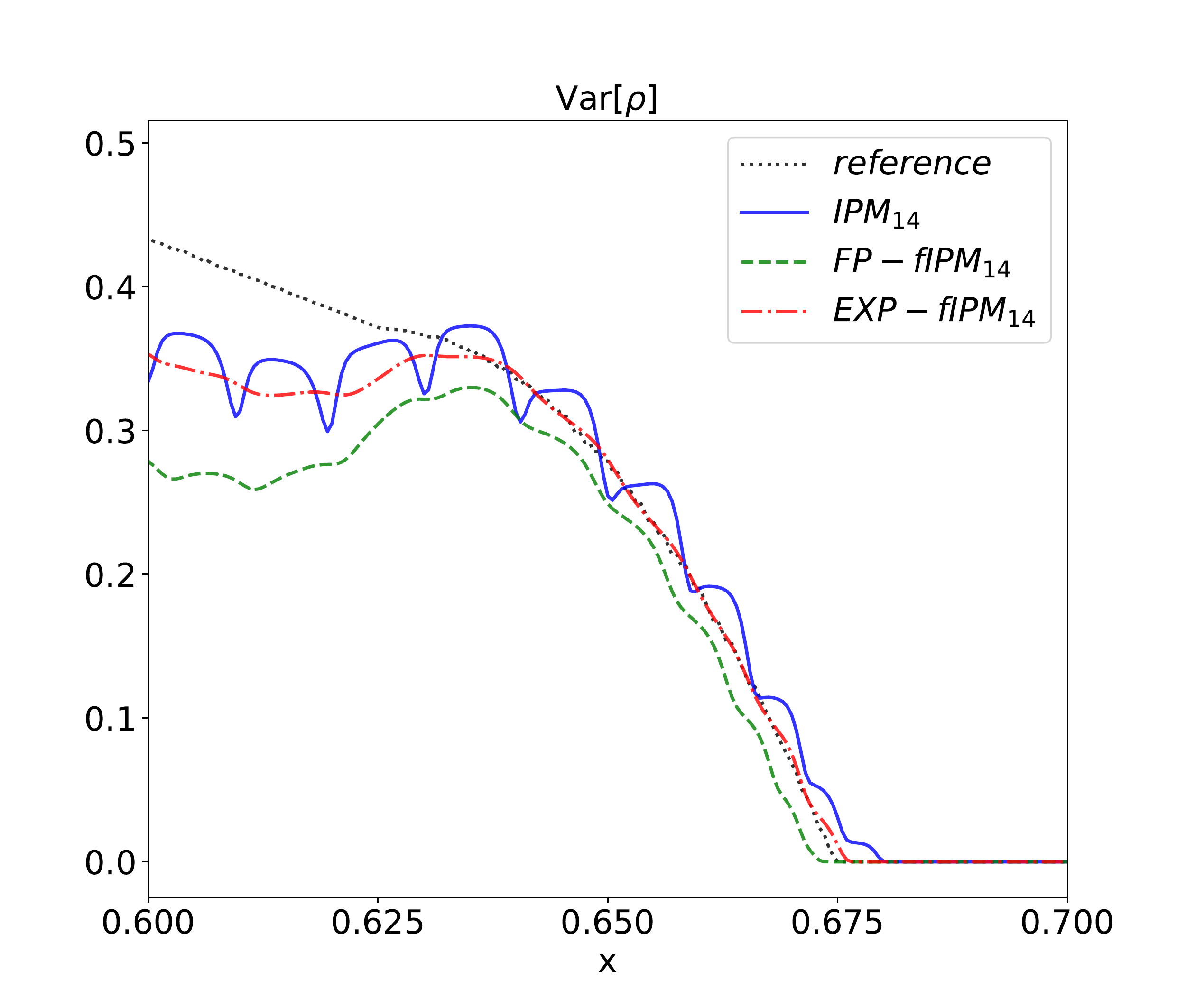}
 \vspace{0.2cm}
    \caption{}
\end{subfigure}
\caption{Expectation and variance for the high density shock computed with IPM and filtered IPM 
($\lambda=2,\alpha=10$ for exponential filtering and $\lambda=10^{-5}$ for 
Fokker-Planck filtering). The exact expected value is depicted in red, the 
exact variance in blue. (a) Expected density. (b) Zoomed view on shock for expected density. (c) Variance of density. (d) Zoomed view on shock for variance of density.}
\label{fig:compareFiltersDensityHD}
\end{figure}

Again, the reference solution has been computed with the analytic solution. Expected value and variance of the analytic solution are approximated by taking a 100 point Gauss-Legendre quadrature. Compared to the previous setting, the test case shows a large amount of density, which travels to the right hand side. In this test case, the rarefaction wave becomes less dominant. One observes that using the standard IPM method leads to heavy oscillations in the IPM solution of both, the expected value and the variance. Moreover, the variance approximation of IPM is heavily dampened out, which is most likely an artifact from the artificial viscosity of the finite volume method \cite[Chapter~16.1]{leveque1992numerical}. Applying filters to this problem resolves oscillations. Especially the exponential filter captures the behavior of the exact variance nicely. Oscillations in the variance approximation are significantly mitigated, however the filter is not able to reduce the dampening effect the finite volume scheme has on the variance. The Fokker-Planck filter improves the approximation of the expected value. However, while oscillations in the variance are dampened, the Fokker-Planck filter further reduces the value of the variance approximation.

\subsection{Filtering for Euler 2D}
\label{sec:2DResults}

\begin{figure}[h!]
\centering
	\begin{subfigure}{1.0\linewidth}
		\centering
		\includegraphics[scale=0.32]{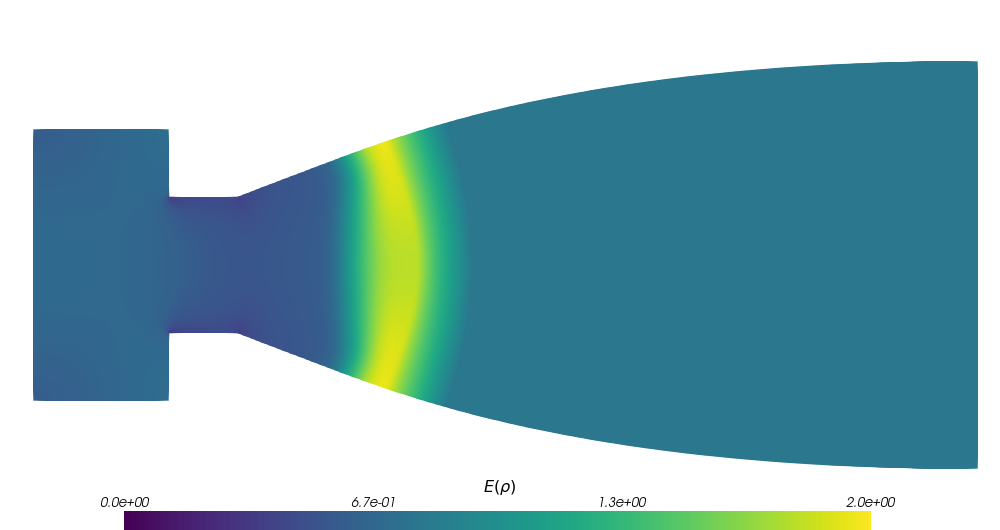}
		\label{fig:sub1}
	\end{subfigure}
	\begin{subfigure}{1.0\linewidth}
		\centering
		\includegraphics[scale=0.32]{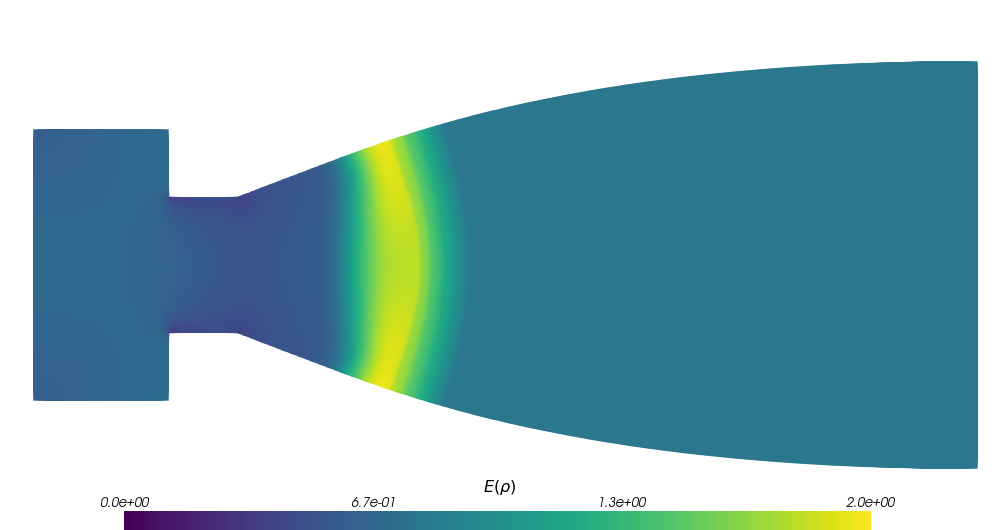}
		\label{fig:sub1}
	\end{subfigure}
	\begin{subfigure}{1.0\linewidth}
		\centering
		\includegraphics[scale=0.32]{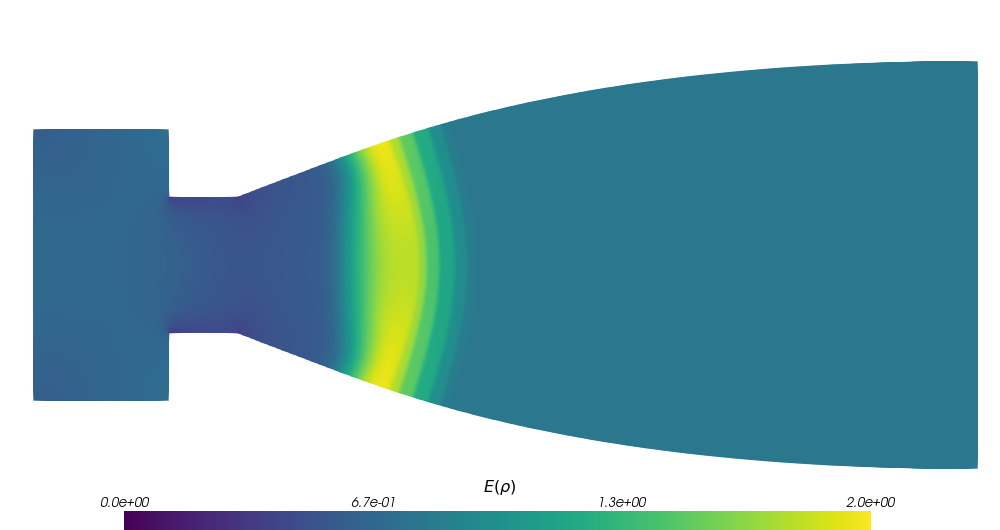}
		\label{fig:sub2}
	\end{subfigure}
	\caption{Expected density with different methods. From top to bottom: reference solution, fIPM, IPM.}
	\label{fig:ExpNozzle2DfIPM}
\end{figure}

\begin{figure}[h!]
\centering
	\begin{subfigure}{1.0\linewidth}
		\centering
		\includegraphics[scale=0.32]{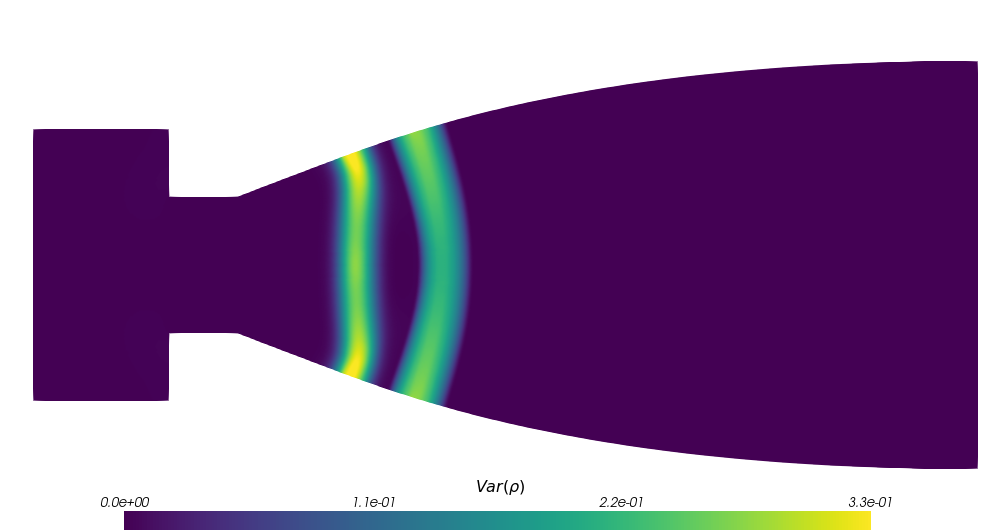}
		\label{fig:sub1}
	\end{subfigure}
	\begin{subfigure}{1.0\linewidth}
		\centering
		\includegraphics[scale=0.32]{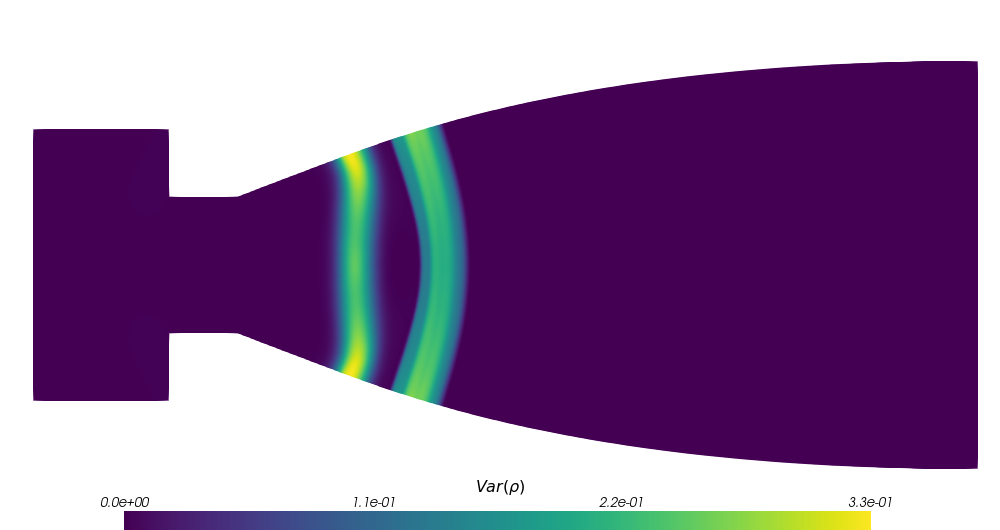}
		\label{fig:sub1}
	\end{subfigure}
	\begin{subfigure}{1.0\linewidth}
		\centering
		\includegraphics[scale=0.32]{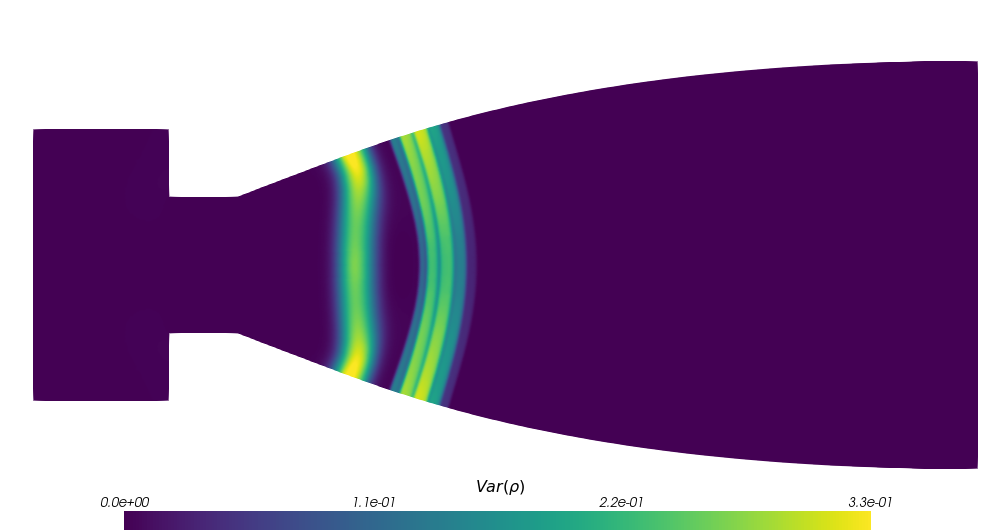}
		\label{fig:sub2}
	\end{subfigure}
	\caption{Variance of the density with different methods. From top to bottom: reference solution, fIPM, IPM.}
	\label{fig:VarNozzle2DfIPM}
\end{figure}
In the following, we will move to the two-dimensional Euler equations 
\begin{align*}
\partial_t
\begin{pmatrix}
\rho \\ \rho v_1 \\ \rho v_2 \\ \rho E
\end{pmatrix}
+\partial_{x_1}
\begin{pmatrix}
\rho v_1 \\ \rho v_1^2 +p \\ \rho v_1 v_2 \\  v_1 (\rho E+p)
\end{pmatrix}
+\partial_{x_2}
\begin{pmatrix}
\rho v_2 \\ \rho v_1 v_2 \\ \rho v_2^2+p \\ v_2 (\rho E+p)
\end{pmatrix}
=\bm{0},
\end{align*}
where the pressure is given by 
\begin{align*}
p = (\gamma-1)\rho\left(E-\frac12(v_1^2+v_2^2)\right).
\end{align*}
We will be interested in quantifying the uncertainty of a shock inside a two-dimensional nozzle. The nozzle geometry is composed of a chamber on the left, a throat in the center and the main nozzle region on the right. As in Sod's shock tube experiment, the gas is initially at rest with an uncertain shock occurring inside the domain. The initial condition is chosen to be
\begin{align*}
\rho_{\text{IC}} &= \begin{cases} \rho_L &\mbox{if } x < 
x_{\text{interface}}(\xi) \\
\rho_R & \mbox{else } \end{cases} \\
(\rho \bm v)_{\text{IC}} &= \bm 0 \\
(\rho E)_{\text{IC}} &= \begin{cases} \rho_L E_L &\mbox{if } x < 
x_{\text{interface}}(\xi) \\
\rho_R E_R & \mbox{else } \end{cases}
\end{align*}
and we pick the shock states as in Section~\ref{sec:resultsHighDensity1D}, i.e., the density on the right hand side is set to $\rho_R=0.8$. The shock position  $x_{\text{interface}}(\xi)$ is distributed uniformly inside the throat. Here we make the choice $x_{\text{interface}}(\xi)=x_{\text{th}}+\sigma\xi$, where $x_{\text{th}}$ is the center $x$-coordinate of the throat. Since the throat has a length of 1, i.e., it ranges from $x_{\text{th}}-0.5$ to $x_{\text{th}}+0.5$, we choose $\sigma = 0.5$ and again use a uniformly distributed $\xi\sim U([-1,1])$. A reference solution, depicted on the top of Figures~\ref{fig:ExpNozzle2DfIPM} and \ref{fig:VarNozzle2DfIPM}, has been computed using Collocation (see \cite{xiu2005high}) with 100 Gauss-Legendre quadrature points. The solution shows an uncertain shock wave which travels to the right side of the nozzle.

In the two-dimensional setting IPM uses the entropy
\begin{align*}
s(\rho,\rho v_1,\rho v_2, \rho E) = -\rho \ln \left(\rho^{-\gamma} \left(\rho E - \frac{(\rho v_1)^2 + (\rho v_2)^2}{2
\rho}\right)\right).
\end{align*}
One can for example find the resulting IPM ansatz in \cite[Appendix~B]{kusch2020intrusive}. The IPM solution, which is depicted in Figures~\ref{fig:ExpNozzle2DfIPM} and \ref{fig:VarNozzle2DfIPM} on the bottom again shows a step-like approximation of the expectation. Furthermore, it leads to an oscillatory variance approximation at the shock position. Let us now apply the Erfc filter \eqref{eq:erfcFilter} to mitigate the spurious artifacts. Since this filter does not preserve realizability, we again need to add a regularization term with $\eta=10^{-7}$, i.e., we choose Algorithm~\ref{alg:fIPM-reg}. The chosen filter strength is $\lambda=0.9$ and we pick a filter order of $\alpha=7$. It can be seen that the resulting approximation shows a satisfactory agreement with the reference solution. The filtered solution shows small artifacts in the shock region, however the overall approximation quality is significantly improved. The exponential filter \eqref{eq:expFilter} shows a similar result compared to the Erfc filter when choosing a filter strength of $\lambda=0.045$ for the same order of $\alpha=7$.\\ 
Now, let us study all filters used in this work, namely the exponential, Erfc and Fokker-Planck filters when changing the filter strength. In order to measure oscillations, we extend the measures $\delta_{\text{E}}$ and $\delta_{\text{Var}}$ to two-dimensional spatial domains: If the distance of a discrete numerical density solution $\rho_{\Delta}:=(\rho_1,\cdots,\rho_{N_x})^T$ to a reference density $\rho_{\text{ex},\Delta}:=(\rho_{\text{ex},1},\cdots,\rho_{\text{ex},N_x})^T$ is given by $e_{\Delta} := \rho_{\text{ex},\Delta}-\rho_{\Delta}$, we are interested in determining the error of second derivatives, which is
\begin{align*}
\delta_{\text{E}}:=\sqrt{\sum_{j=1}^{N_x} \Delta x_j\left( \left( \partial_{xx} \mathbb{E}[e_j] \right)^2 + \left( \partial_{yy} \mathbb{E}[e_j] \right)^2 \right)}
\end{align*}
and
\begin{align*}
\delta_{\text{Var}}:=\sqrt{\sum_{j=1}^{N_x} \Delta x_j\left( \left( \partial_{xx} \text{Var}[e_j] \right)^2 + \left( \partial_{yy} \text{Var}[e_j] \right)^2 \right)}.
\end{align*}
With this error measure, we intend to capture the error behavior while measuring oscillations. Note that this measure only gives a coarse overview and minimal values in $\delta_{\text{E}}$ and $\delta_{\text{Var}}$ do not always guarantee obtaining the optimal result when looking at the full expectation and variance plot. The resulting values of $\delta_{\text{E}}$ and $\delta_{\text{Var}}$ for different filter strengths at a fixed filter order $\alpha=7$ (when using the EXP and erfc filters) are depicted in Figure~\ref{fig:ErrorsDelta}.  Here, we added the values of $\delta_{\text{E}}$ and $\delta_{\text{Var}}$ of the classical IPM method without filters as straight lines. Note that all filters require differently chosen filter strengths and we therefore study these filters on different grids for $\lambda$. For the order $7$ filters, these grids are chosen such that the effects of filtering with each filter is plausible, i.e., the solution does not equal the standard IPM solution and the effects of uncertainty is not completely dampened out. The Fokker-Planck filter in Figure~\ref{fig:ErrorsDeltasub3} is very sensitive with respect to the chosen filter strength and we can therefore observe the effect of filtering out the dependency on $\xi$ for large filter strengths. Note that in this case, the variance has a value of zero. Hence, $\delta_{\text{Var}}$ is simply the second derivative of the reference solution. Since the reference solution is smooth, this value is small compared to settings in which the filter is less dominant. However, the poor quality of the solution can be read from $\delta_{\text{E}}$ which heavily increases. At intermediate filter strengths, the Fokker-Planck filter does show improved results in terms of the chosen oscillation metric. It can be seen that both, the exponential filter in Figure~\ref{fig:ErrorsDeltasub1} and the Erfc filter in Figure~\ref{fig:ErrorsDeltasub2} mitigate oscillations compared to the classical IPM solution. However, the choice of the filtering strength affects the quality of the expected value and variance. While the optimal filter strength differs for the expected value and the variance, we observe that the minimum in the respective curves for $\delta_{\text{E}}$ and $\delta_{\text{Var}}$ is reached at similar points. Therefore, if we choose the filter such that we obtain a good representation of the expected value, we can expect to also observe a satisfactory variance approximation. 
\begin{figure}[h!]
\centering
	\begin{subfigure}{0.49\linewidth}
		\centering
		\includegraphics[scale=0.32]{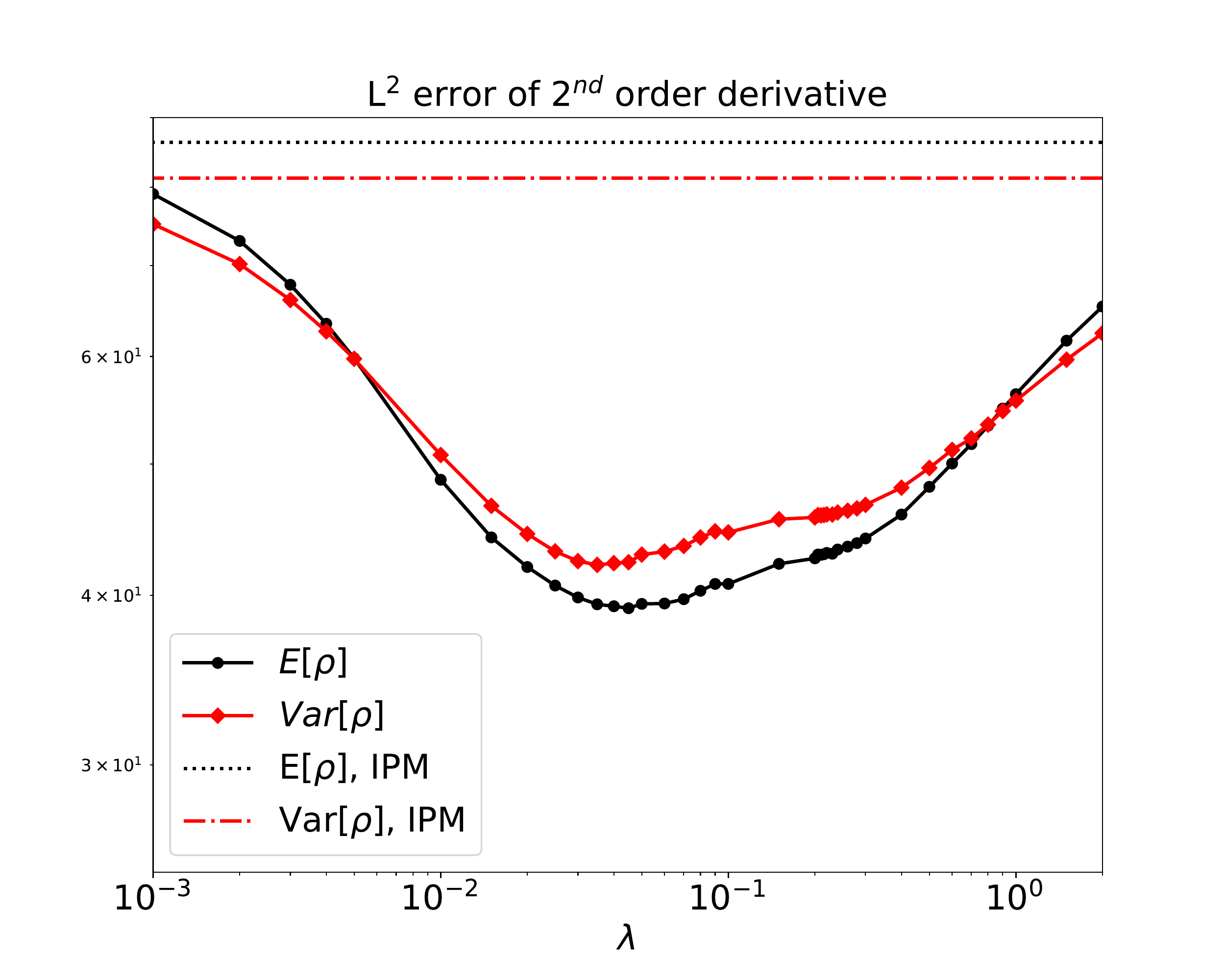}
		\caption{Exponential filter}
		\label{fig:ErrorsDeltasub1}
	\end{subfigure}
	\begin{subfigure}{0.49\linewidth}
		\centering
		\includegraphics[scale=0.32]{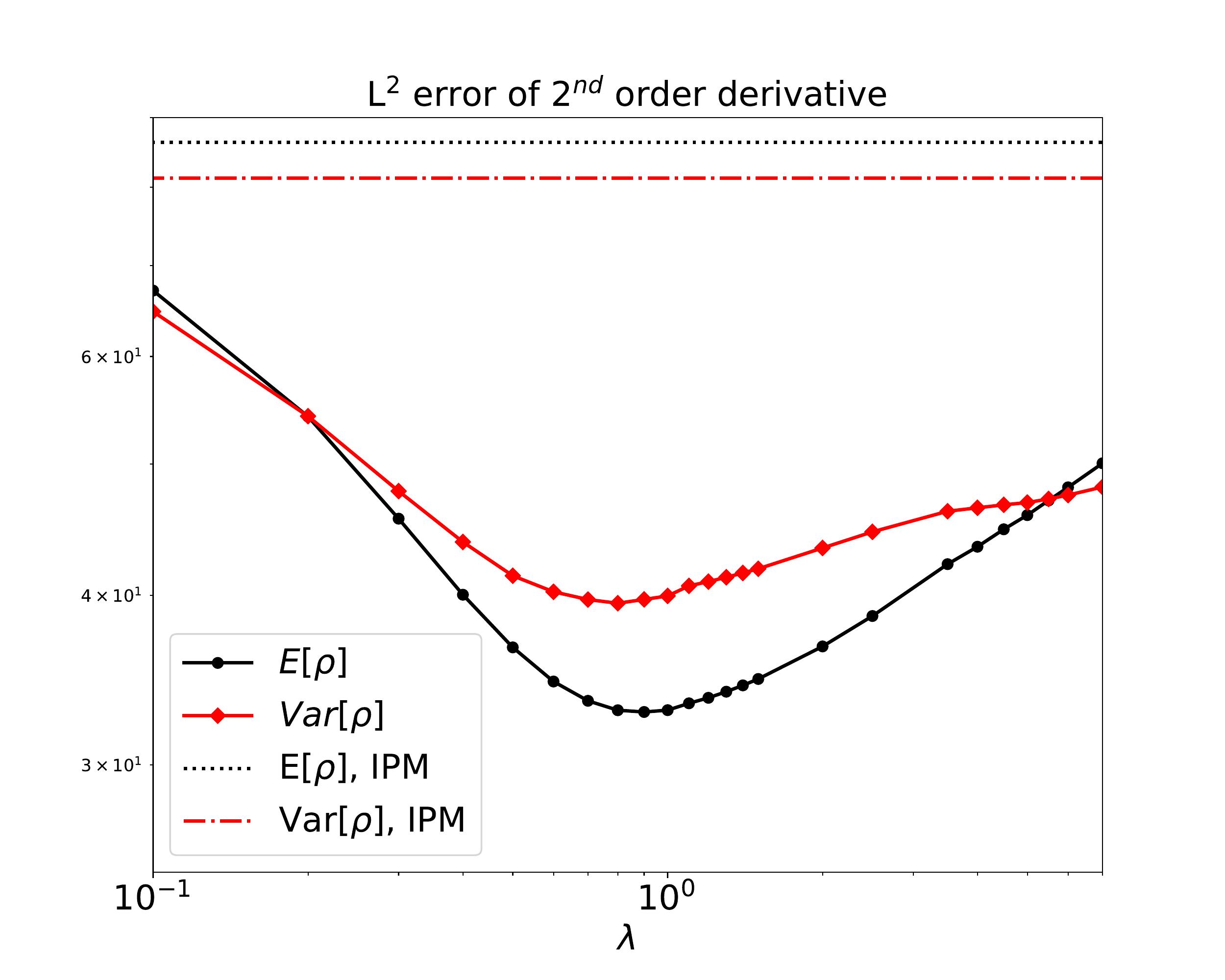}
		\caption{Erfc filter}
		\label{fig:ErrorsDeltasub2}
	\end{subfigure}
		\begin{subfigure}{0.49\linewidth}
		\centering
		\includegraphics[scale=0.32]{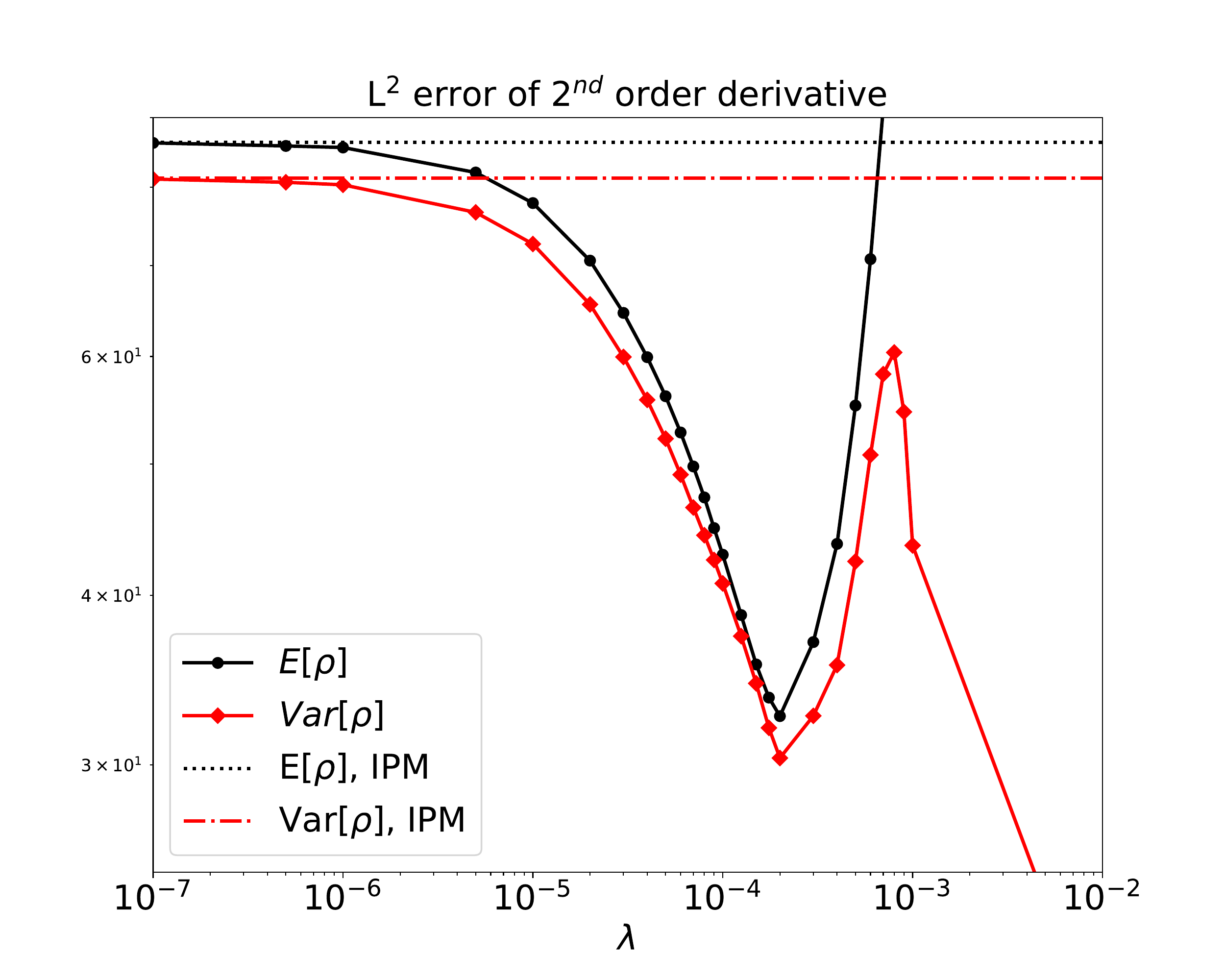}
		\caption{Fokker-Planck filter}
		\label{fig:ErrorsDeltasub3}
	\end{subfigure}
	\caption{$\delta_{\text{E}}$ and $\delta_{\text{Var}}$ for different filter strengths when using the exponential, Erfc and Fokker-Planck filter. The exponential and erfc filters have a regularization strength $\eta=10^{-7}$ and order $\alpha=7$. Values for $\delta_{\text{E}}$ and $\delta_{\text{Var}}$ without filters are added as a straight line.}
	\label{fig:ErrorsDelta}
\end{figure}

\section{Summary and Outlook}
\label{sec:outlook}
The main task of this work is to develop a method which mitigates oscillations that arise in the solutions of the IPM method. Applying a filter to the moments of IPM is challenging, since the filtered moments need to remain realizable and it can be shown that standard filters violate this property. In this work, we presented two different strategies to apply filtering to the IPM method. First, we proposed a new filter, which we called Fokker-Planck filter. It is based on an underlying Fokker-Planck equation, whose solutions obey the maximum--principle. We use this fact to guarantee realizability of filtered moment vectors. Furthermore, by investigating the filter on the kinetic level of the Euler equations, we are able to show realizability for applications in fluid dynamics. Second, we proposed to modify the IPM closure to be able to cope with non-realizable moments. This task is achieved by adding a regularization term to the IPM optimization problem. The regularization then allows choosing standard filters, which do not necessarily preserve realizability.
We conduct numerical experiments to investigate the effectiveness of both approaches. It turns out that the second approach, i.e., using a regularization strategy yields the best results. I.e., the resulting approximations of expected value and variance show a significant reduction of artifact which resulted from spurious oscillations. Unfortunately, the Fokker-Planck filter, despite being able to improve the solution at the shock, heavily dampens the variance at the rarefaction wave. 

Ideas to weaken this effect are the use of adaptive filter strengths, which 
yields the possibility to turn down the filter strength at the rarefaction wave 
and contact discontinuity. Note, that the classical IPM method without filters 
yields good results in theses regions, meaning that one could only turn on 
filtering at the shock position. In this work, we focused on the exponential, Erfc and 
Fokker-Planck filter and one should try out more filter functions. Note that 
the Lasso filter \cite{kusch2018filtered}, which adaptively picks the filter 
strength in an automated fashion, unfortunately led to heavily dampened 
variance results. 
A method which is closely related to IPM is the use of Roe variables \cite{pettersson2014stochastic,gerster2020entropies}. Applying the proposed filtering and regularization strategies to Roe variables could potentially improve this class of methods.
Furthermore, instead of regularizing the moment system, one could apply the regularization on the Euler equations as done in \cite{alldredge2018regularized}. Such a system allows the use of non-physical quantities such as negative densities and pressures. In this case, one is able to apply the standard stochastic-Galerkin method to the regularized Euler equations without having to preserve positivity of certain physical quantities.
The presented ideas are not restricted to uncertainty quantification as they can for 
example be used in kinetic theory and other research areas which require the 
construction of closures. Furthermore, one should further investigate the 
regularization, especially its effect on the run time. The observed run time 
speedup when increasing the regularization parameter holds the potential of 
speeding up standard IPM computations. However, one needs to keep in mind that 
the numerical experiments show strong solution manipulations when the 
regularization is chosen too big. 

\section*{Acknowledgments}
Funding: Jonas Kusch, Graham Alldredge and Martin Frank were supported by the 
German Research Foundation (DFG). Martin Frank was supported 
under grant FR 2841/6-1 and Graham Alldredge under AL 2030/1-1. Jonas Kusch was funded under Project-ID 258734477 - SFB 1173. Furthermore, 
Jonas Kusch would like to thank the German Academic Exchange Service (DAAD) for 
funding his research visit at the University of Notre Dame as well as the University of Notre Dame for hosting his stay during which large portions of this work were conducted.

\bibliographystyle{unsrt}  
\bibliography{references}

\appendix
\section{Sturm-Liouville equation of orthogonal polynomials}\label{app:orthoPoly}
In the following, we provide a short summary of the Sturm-Liouville equation, following \cite[Chapter~5]{powers2015mathematical}. The standard orthonorormal polynomials $P_i:\Theta\rightarrow\mathbb{R}$ fulfill a Sturm-Liouville equation. This equation takes the form $\mathcal{L} P_i = \mu_i P_i$, where $\mathcal{L}$ is a self-adjoint linear operator
\begin{align*}
    \mathcal{L}P_i = \frac{1}{r(x)}\frac{d}{d\xi}\left(p(x)\frac{d}{d\xi}+q(x)\right).
\end{align*}
I.e., the polynomials $P_i$ are eigenfunctions to the operator $\mathcal{L}$ with eigenvalues $\mu_i$. Furthermore, we have that
\begin{align*}
    \int_{\Theta} P_i(\xi) P_j(\xi) r(\xi)\,d\xi = \gamma_i \delta_{ij},
\end{align*}
where $\gamma_i$ denotes the weighted $L^2$-norm.
The \textit{Legendre polynomials} which are used throughout this work fulfill this property, where
\begin{align*}
    p(\xi) = 1-\xi^2, \quad r(\xi) = 1, \quad q(\xi) = 0, \quad \mu_i = -i(i+1).
\end{align*}
Further polynomials are
\begin{itemize}
    \item \textit{Chebyshev}: $p(\xi) = \sqrt{1-\xi^2}, \quad r(\xi) = \frac{1}{\sqrt{1-\xi^2}}, \quad q(\xi) = 0, \quad \mu_i = -i^2$,
    \item \textit{Hermite}: $p(\xi) = e^{-\xi^2}, \quad r(\xi) = e^{-\xi^2}, \quad q(\xi) = 0, \quad \mu_i = -2i$,
    \item \textit{Laguerre}: $p(\xi) = \xi e^{-\xi^2}, \quad r(\xi) = e^{-\xi^2}, \quad q(\xi) = 0, \quad \mu_i = -i$.
\end{itemize}
\end{document}